\documentclass[reqno,11pt]{article}
\usepackage{amsmath,amsfonts,amsthm,geometry,import,mathrsfs,color,dsfont,comment,amssymb}

\newcommand{\R}{\mathbb{R}}

\newcommand{\RD}{\mathbb{R}^d}
\newcommand{\N}{\mathbb{N}}

\newcommand{\cF}{\mathcal{F}}
\newcommand{\cE}{\mathcal{E}}

\newcommand{\cH}{\mathcal{H}}

\newcommand{\cP}{\mathcal{P}}

\newcommand{\scrA}{\mathscr{A}}

\newcommand{\bes}{\begin{equation*}}
\newcommand{\ees}{\end{equation*}}
\newcommand{\beas}{\begin{eqnarray*}}
\newcommand{\eeas}{\end{eqnarray*}}
\newcommand{\bea}{\begin{eqnarray}}
\newcommand{\eea}{\end{eqnarray}}
\newcommand{\be}{\begin{equation}}
\newcommand{\ee}{\end{equation}}
\newcommand{\bei}{\begin{itemize}}
\newcommand{\eei}{\end{itemize}}
\newcommand{\bec}{\begin{cases}}
\newcommand{\eec}{\end{cases}}
\newcommand{\ben}{\begin{enumerate}}
\newcommand{\een}{\end{enumerate}}

\newcommand{\ve}{\varepsilon}

\newcommand{\bbE}{\mathbb{E}}

\newcommand{\bbl}{\begin{block}}
\newcommand{\ebl}{\end{block}}

\newcommand{\De}{\mathrm{d}}

\newcommand{\rmP}{\mathrm{P}}
\newcommand{\rmQ}{\mathrm{Q}}

\newtheorem{prop}{Proposition}[section]
\newtheorem{theorem}[prop]{Theorem}
\newtheorem{lemma}[prop]{Lemma}

\newtheorem{remark}[prop]{Remark}

\newtheorem{cor}[prop]{Corollary}
\newcommand{\cost}{\mathscr{C}_T}
\newcommand{\mfcost}{\mathscr{C}^{mf}_T}
\newcommand{\ent}{\cH}
\newcommand{\mm}{\mathfrak{m}}
\newcommand{\cons}{\mathscr{E}_T}

\newcommand{\CD}{{\sf CD}}
\newcommand{\hp}{{\sf p}}
\newcommand{\sfd}{{\sf d}}
\newcommand{\supp}{\mathop{\rm supp}\nolimits} 
\newcommand{\leb}{\lambda} 
\newcommand{\sfL}{{\sf L}}
\newcommand{\sfP}{{\sf P}}
\newcommand{\sfQ}{{\sf Q}}

\newcommand{\fr}{\penalty-20\null\hfill$\blacksquare$}
\usepackage{a4wide}
\usepackage[pdfborder={0 0 0}]{hyperref}  

\allowdisplaybreaks

\numberwithin{equation}{section}

\title{A formula for the time derivative of the entropic cost and applications}
\author{Giovanni Conforti \thanks{D\'epartement de Math\'ematiques Appliqu\'ees, \'Ecole Polytechnique, Route de Saclay, 91128 Palaiseau Cedex, France. email: giovanni.conforti@polytechnique.edu} \quad Luca Tamanini \thanks{CEREMADE (UMR CNRS 7534), Universit\'e Paris Dauphine PSL, Place du Mar\'echal de Lattre de Tassigny, 75775 Paris Cedex 16, France and INRIA-Paris, MOKAPLAN, 2 Rue Simone Iff, 75012, Paris, France. email: tamanini@ceremade.dauphine.fr}}

\begin{document}

\maketitle

\begin{abstract}
In the recent years the Schr\"odinger problem has gained a lot of attention because of the connection, in the small-noise regime, with the Monge-Kantorovich optimal transport problem. Its optimal value, the \emph{entropic cost} $\cost$, is here deeply investigated. In this paper we study the regularity of $\cost$ with respect to the parameter $T$ under a curvature condition and explicitly compute its first and second derivative. As applications:
\begin{itemize}
\item[-] we determine the large-time limit of $\cost$ and provide sharp exponential convergence rates; we obtain this result not only for the classical Schr\"odinger problem but also for the recently introduced Mean Field Schr\"odinger problem \cite{backhoff2019mean};
\item[-] we improve the Taylor expansion of $T \mapsto T\cost$ around $T=0$ from the first to the second order.
\end{itemize}
\end{abstract}

\tableofcontents

\section{Introduction and statement of the main results}\label{sec:stat main res}

The entropic transportation cost is the optimal value in a probabilistic version of the Monge-Kantorovich optimal transport problem, the Schr\"odinger problem, whose study has already shown to have far reaching consequences in various fields, ranging from statistical machine learning to functional inequalities. The goal of the present article is to advance in the study of the entropic cost as a function of the time (regularization) parameter $T$ (see Definition \ref{cost def} below). Following Mikami's contribution \cite{Mikami04} linking the Schr\"odinger problem to optimal transport, several results have been obtained in the last years concerning the behavior of the entropic cost in the short-time (small noise) limit $T\rightarrow 0$. In particular, as a byproduct of the research line originated in \cite{Adams2011} another step forward was made with the computation of the first derivative of the rescaled entropic cost at $T=0$, see also \cite{erbar2015from}, the recent works \cite{pal2019difference,di2019optimal} and references therein. However, very few results beyond the short-time limit have been obtained. In particular, very little is known about the long-time regime $T \rightarrow \infty$, where the entropic cost is expected to converge to the sum of the marginal entropies. This lack of knowledge was one of the main motivations for our work and in this respect, our contribution includes:
\begin{itemize}
\item A formula for the first and second derivative of the (rescaled) entropic cost for a general value of $T$ in terms of the so-called ``energy" (defined at \eqref{eq:constant quantity} below).
\item A rigorous identification of the large-time limit of the entropic cost as the sum of the marginal entropies.
\item Sharp exponential convergence rates under a curvature condition in the long-time regime. We obtain this result not only for the classical Schr\"odinger problem but also for the recently introduced Mean Field Schr\"odinger problem \cite{backhoff2019mean}.
\end{itemize}
We also establish some results in the short-time limit. Their interest resides in the fact that they allow to obtain a clean expression of the second derivative of the rescaled entropic cost that, to the best of our knowledge, was not known before. However, the regularity assumptions we impose on the marginals are not the weakest ones. For this part, our contribution can be resumed as follows:
\begin{itemize}
\item A formula for the second derivative of the rescaled cost around $T=0$ that yields  the local convexity of the cost in the time variable.
\item A non-asymptotic sharp quantitative bound for the convergence to the Wasserstein distance depending only on the integral of the Fisher information functional along the corresponding Wasserstein geodesic.
\end{itemize}

\subsection{The Schr\"odinger problem}

We consider a triplet $(M,\sfd_g,\mm)$, where $M$ is a smooth, connected and complete Riemannian manifold without boundary and with metric tensor $g$, $\sfd_g$ is the induced distance and $\mm$ is given by $\mm(\De x)= \exp(-2U(x))\mathrm{vol}(\De x)$, where $U$ is a $C^2$ potential satisfying the Bakry-\'Emery $\CD(\kappa,\infty)$ condition, namely
\be
\label{hyp pot}\tag{H1}
U \in C^2(M;\R), \qquad \mathrm{Ric}_g + \mathrm{Hess}(2U) \geq \kappa g \,\, \text{ for some $\kappa\in\R$}. 
\ee
The measure $\mm$ is invariant for the SDE
\begin{equation}
\label{SDE}
\De X_t = -\nabla U(X_t)\De t + \De B_t
\end{equation}
(here $B_t$ denotes a standard Brownian motion) whose joint law at time $0$ and $T$ will be denoted by $R_{0,T}$. Within this framework and given $\mu,\nu \in \cP(M)$ (as usual, for a measurable space $(E,\cE)$ we denote $\cP(E)$ the set of probability measures over $(E,\cE)$), the entropic transportation cost $\cost(\mu,\nu)$ is defined as the optimal value in the corresponding Schr\"odinger problem, namely
\be
\label{cost def}
   \cost(\mu,\nu) = \inf_{\pi \in \Pi(\mu,\nu)} \ent(\pi\,|\,R_{0,T}),  \
\ee
where $\Pi(\mu,\nu) \subset \cP(M \times M)$ is the set of couplings of $\mu$ and $\nu$ and $\cH$ is the relative entropy functional defined for two probability measures $p,q$ on the same (arbitrary) measurable space as
\begin{equation}
\label{eq:entropy}
\cH(p\,|\,q)=\begin{cases}
\displaystyle{\int  \log \Big(\frac{\De p}{\De q }\Big)\De p}, &\quad  \mbox{if $p\ll q$}\\
+\infty, & \quad \text{otherwise}.
\end{cases}
\end{equation}
When $q$ is not a probability, the precise definition is postponed to Section \ref{sec:preliminaries}.

The physical meaning of the variational problem \eqref{cost def} is described in the seminal papers \cite{Schrodinger31,Schrodinger32}, where E.\ Schr\"odinger addressed the problem of finding the most likely evolution of a system of independent random particles driven by \eqref{SDE} conditionally on the observation of their initial and final configuration. With this picture in mind, the short- and long-time behavior of the entropic cost sounds perfectly natural.

\paragraph{Assumptions on $\mm$, $\mu$ and $\nu$.} We state here the additional assumptions we will make either on the reference measure $\mm$ or on the marginals $\mu,\nu$. Concerning the former, we will often assume that
\begin{equation}\label{hyp ref}\tag{H2}
    \mm \text{ is a probability measure}.
\end{equation}
Notice that when $\kappa > 0$, \eqref{hyp ref} automatically holds and in addition $\mm \in \cP_2(M)$ (see for instance \cite[Theorem 4.26]{Sturm06I}). For the latter, we shall suppose that either
\begin{equation}\label{hyp marginals weakest}\tag{H3}
\mu,\nu \in \cP_2(M) \,\textrm{ and }\, \cH(\mu\,|\,\mm),\cH(\nu\,|\,\mm) < +\infty,
\end{equation}
where $\cP_2(M) \subset \cP(M)$ denotes the space of probability measures over $M$ with finite second moment, or, more frequently, that
\be\label{hyp marginals weak}\tag{H4}
\mu = \rho\mm,\,\nu = \sigma\mm \,\textrm{ and }\, \rho,\sigma \in L^\infty(\mm)\,\textrm{ with compact support}.
\ee
Notice that \eqref{hyp marginals weak} is stronger than \eqref{hyp marginals weakest}.

\paragraph{The Benamou-Brenier formulation.} The fluid-dynamical formulation of the entropic cost \cite{ChGePa16,GLR15, GigTam19, Tamanini20} asserts that, at least under \eqref{hyp marginals weak}, we have
\begin{equation}\label{eq:bbs}
\begin{split}
\cost(\mu,\nu) & = \frac12\Big(\ent(\mu\,|\,\mm) + \ent(\nu\,|\,\mm)\Big) + \inf_{(\bar{\rho},\bar{v})}\iint_0^T \Big(\frac12|\bar{v}_t|^2 + \frac18|\nabla\log\bar\rho_t|^2\Big) \rho_t\De t\De\mm \\
& = \frac12\Big(\ent(\mu\,|\,\mm) + \ent(\nu\,|\,\mm)\Big) + \iint_0^T \Big(\frac12|\bar{v}_t^T|^2 + \frac18|\nabla\log\bar\rho_t^T|^2\Big) \bar{\rho}_t^T\De t\De\mm
\end{split}
\end{equation}
where $(\bar{\rho}^T,\bar{v}^T)$ is the unique optimal curve and the infimum runs over all weak solutions $(\bar{\rho},\bar{v})$ of the continuity equation
\begin{equation}
\label{eq:conteq}
\partial_t(\bar{\rho}_t\mm) + {\rm div}_\mm(\bar{v}_t\bar{\rho}_t\mm) = 0,
\end{equation}
satisfying the marginal constraints $\rho_0\mm = \mu$ and $\rho_T\mm = \nu$, namely among all $(\bar{\rho}_t) \subset L^\infty(\mm)$ with $\bar{\rho}_t\mm \in \cP(M)$ and all Borel vector fields $(\bar{v}_t)$ such that: 
\begin{itemize}
\item[(i)] $t \mapsto \bar{\rho}_t\mm$ is weakly continuous and there exists $C>0$ such that $\bar{\rho}_t \leq C$ for all $t \in [0,T]$; 
\item[(ii)] the map $t \mapsto \int |\bar{v}_t|^2\bar{\rho}_t\De\mm$ is Borel and belongs to $L^1(0,T)$; 
\item[(iii)] for any $f \in C_c^\infty(M)$ the map $[0,T] \ni t \mapsto \int f\bar{\rho}_t\,\De\mm$ is absolutely continuous and it holds
\[
\frac{\De}{\De t}\int f\bar{\rho}_t\,\De\mm = \int \De f(\bar{v}_t)\bar{\rho}_t\,\De\mm \qquad {\rm a.e.}\ t.
\]
\end{itemize}
In the above, we denoted by ${\rm div}_\mm$ the divergence w.r.t.\ $\mm$ and by $\De f$ the differential of $f$. 

From a physical point of view, \eqref{eq:bbs} states that the trajectory $(\bar{\rho}_t^T)_{t \in [0,T]}$, also called \emph{entropic interpolation}, is the one minimizing a functional consisting of two terms: the former is purely kinetic, while the latter is given by the Fisher information functional. Hence $(\bar{\rho}_t^T)_{t \in [0,T]}$ is given by the balance between deterministic and chaotic behavior.

\paragraph{The energy.} The quantity
\begin{equation}\label{eq:constant quantity}
\frac12 \int |\bar{v}^T_t|^2 \bar{\rho}^T_t \De \mm - \frac18 \int |\nabla \log \bar{\rho}^T_t|^2  \bar{\rho}^T_t \De \mm
\end{equation}
is conserved along the optimal flow $(\bar{\rho}_t^T)_{t \in [0,T]}$, i.e.\ its value does not depend on $t$, see \cite[Corollary 1.1]{Conforti17} and \cite[Lemma 3.2]{GLRT19} (although in these papers the setting is more restrictive, the argument therein carries over verbatim to the present setting). We call this constant the \emph{energy} and denote it $\cons(\mu,\nu)$. In order to justify the term ``energy'', we remind that the formal Riemannian calculus on $\cP_2(M)$ introduced by Otto \cite{Otto01} allows to draw an analogy between the minimization \eqref{eq:bbs} and the much simpler problem (see \cite{GLR18}) on $\RD$

\be\label{eq:toySchr}
\inf_{\substack{x:[0,T]\rightarrow\RD\\ x_0=x,x_T=y}} \int_{0}^T \frac{1}{2}|\dot{x}_t|^2 + \frac{1}{8}|\nabla f|^2(x_t)\, \De t,
\ee
where the role of the function $f$ in \eqref{eq:toySchr} is played by the relative entropy $\cH(\cdot|\mm)$ in \eqref{eq:bbs}. For \eqref{eq:toySchr} it is well known that along any critical curve the total energy (given by the sum of the kinetic and potential components)
\bes
\frac{1}{2}|\dot{x}_t|^2 - \frac{1}{8}|\nabla f|^2(x_t)
\ees
is conserved. The analogous of this simple fact for the problem \eqref{eq:bbs} is the conservation of \eqref{eq:constant quantity} along optimal flows.

\paragraph{Short-time behavior of the entropic cost.} 

The fact that the rescaled entropic cost converges to the squared Wasserstein distance of order two in the short-time limit, i.e.
\be\label{cost to wass}
\lim_{T\rightarrow 0} T \, \cost(\mu,\nu) = \frac12 W_2^2(\mu,\nu),
\ee
has generated a surge of interest around the Schr\"odinger problem (henceforth SP), since SP is more regular and numerically easier to solve than the Monge-Kantorovich problem. There exist nowadays several proofs of \eqref{cost to wass}, see for instance \cite{Leonard12,Mikami04} for $\Gamma$-convergence results. A proof that is valid under the  hypotheses \eqref{hyp pot} and \eqref{hyp marginals weak} can be found in \cite[Remark 5.11]{GigTam18}.
In \cite{Adams2011}, (see also \cite{duong2013wasserstein, erbar2015from, pal2019difference}) a further fundamental step was taken that consists in computing the first order term in the expansion of $T\mathscr{C}_T(\mu,\nu)$ around $T=0$. Referring to the above mentioned articles for precise statements, we have the following expansion\footnote{In the above mentioned references, one often finds $\cH(\mu\,|\,\mm)-\cH(\nu\,|\,\mm)$ instead of $\cH(\mu\,|\,\mm)+\cH(\nu\,|\,\mm)$ in the first order term. This is due to a slightly different choice of reference measure $R_{0,T}$. Since we prefer the reference measure to be reversible, we gain an extra $\cH(\mu\,|\,\mm)$ term in the Taylor expansion.}
\be\label{eq:1storder}
T\mathscr{C}_T(\mu,\nu) = \frac{1}{2}W_2^2(\mu,\nu) + \frac{T}{2} (\cH(\mu\,|\,\mm) + \cH(\nu\,|\,\mm)) +o(T).
\ee
In this work we establish at Theorem \ref{thm short time} a non-asymptotic bound for $T\mathscr{C}_T(\mu,\nu) - W_2^2(\mu,\nu)/2$ which is sharp in the limit $T \to 0$ and we compute the second order term in \eqref{eq:1storder}, thus getting
\begin{equation}
\label{eq:2nd order}
T\mathscr{C}_T(\mu,\nu) = \frac{1}{2}W_2^2(\mu,\nu) + \frac{T}{2}(\cH(\mu\,|\,\mm)+\cH(\nu\,|\,\mm))+\frac{T^2}{8} \iint_0^1 |\nabla \log \rho^0_t|^2\rho^0_t \,\De t \De\mm + o(T^2),
\end{equation}
where $(\rho^0_t \mm)_{t\in[0,1]}$ denotes the (unique) Wasserstein geodesic between $\mu$ and $\nu$. Let us remark that \eqref{eq:2nd order} tells that the rescaled cost is convex around $T=0$ and using functional inequalities such as the HWI inequality \cite{OttoVillani00} one can also estimate from below its second derivative under the Bakry-\'Emery condition \eqref{hyp pot}. It is an interesting question to obtain general versions of \eqref{eq:2nd order} in terms of $\Gamma$-convergence. In addition, it is worth noticing that the energy, once properly rescaled, also converges to the squared Wasserstein distance, namely 
\[
\lim_{T\rightarrow 0} T^2 \cons(\mu,\nu) = \frac{1}{2} W_2^2(\mu,\nu).
\]
To see, at least formally, why this is expected to be true, one can integrate \eqref{eq:constant quantity} in time, use \eqref{eq:bbs} and the energy conservation to get 
\bes
T\cons(\mu,\nu) = \cost(\mu,\nu) - \frac12\Big(\ent(\mu\,|\,\mm) + \ent(\nu\,|\,\mm)\Big) - \frac{1}{4}\iint_{0}^T |\nabla \log \bar{\rho}^T_t|^2 \bar{\rho}^T_t\, \De t\De \mm
\ees
Multiplying by $T$, letting $T\rightarrow 0$ and using \eqref{cost to wass} we can then handle the first term on the right-hand side as well as $\ent(\mu\,|\,\mm) + \ent(\nu\,|\,\mm)$. The fact that the last integral converges to $0$ has been proved in \cite[Lemma 3.3]{GLRT19} and the same argument can be adapted verbatim to the present setting (see \cite{Tamanini20}); yet it is a non-trivial fact, as the entropic interpolation $(\bar{\rho}^T_t)_{t \in [0,T]}$ converges to the Wasserstein geodesic $(\rho^0_t)_{t\in[0,T]}$ between $\mu$ and $\nu$, whose Fisher information needs not to be defined.

\paragraph{Long-time behavior of the entropic cost.} 
Although its asymptotic regime has been the object of recent studies in connection with the ergodic behavior of Schr\"odinger bridges \cite{backhoff2019mean} and the limiting behavior of Sinkhorn divergence \cite{feydy2018interpolating}, very few results are available. In this article we prove that under mild assumptions we have
\begin{equation}
\label{cost to entropy}
\lim_{T\rightarrow +\infty}  \cost(\mu,\nu) = \ent(\mu \,|\,\mm)+\ent(\nu\,|\,\mm),
\end{equation}
The intuition behind \eqref{cost to entropy} is that $R_{0,T}$ converges towards $\mm\otimes\mm$ and therefore the optimal coupling in SP converges towards the independent coupling of $\mu$ and $\nu$. We remark that in a different although related context \cite{feydy2018interpolating}, the convergence of Sinkhorn divergences towards MMD divergences in the limit when the regularization parameter goes to $+\infty$ shares many analogies with \eqref{cost to entropy}.

Our second main result deals with the approximation error in \eqref{cost to entropy} and states that it is exponentially small (see Theorem \ref{thm cost bound} for the rigorous statement)
\be\label{long time intro}
|\cost(\mu,\nu) - \ent(\mu \,|\,\mm)-\ent(\nu\,|\,\mm)| \lesssim \exp\Big(-\frac{\kappa}{2} T\Big)
\ee
provided the Bakry-\'Emery condition \eqref{hyp pot} is satisfied. The proof of \eqref{long time intro} is based on Theorem \ref{thm:derivcost} asserting that the time derivative of $\cost(\mu,\nu)$ is precisely $-\mathscr{E}_{T}(\mu,\nu)$ and on two functional inequalities. The first one is a version of the Talagrand inequality (obtained in  \cite{Conforti17} and also called entropic Talagrand inequality) and the second one is a functional inequality relating $|\mathscr{E}_{T}(\mu,\nu)|$ with $\cost(\mu,\nu)$, that we call ``energy-transport" inequality (cf.\ Lemma \ref{lem:cons cost}). A similar inequality has been proved very recently in \cite{backhoff2019mean}; there it has been used to obtain the so-called \emph{turnpike property} for mean field Schr\"odinger bridges.
It is worth noticing that estimates such as \eqref{long time intro} do not seem to follow from classical functional inequalities such as Talagrand and Log-Sobolev, whereas they can be proven using the new family of functional inequalities involving the entropic cost $\cost(\mu,\nu)$. The final contribution of the article is to derive a bound similar to \eqref{long time intro} for the Mean Field Schr\"odinger problem introduced in \cite{backhoff2019mean}. Since we could not establish a generalization of the differentiation formula at Theorem \ref{thm:derivcost} to the mean field setup, the proof of this estimate follows a different scheme, but is still based on a class of functional inequalities derived in \cite{backhoff2019mean} that are the mean field versions of the entropic Talagrand and of the energy-transport inequalities mentioned above.
\paragraph{Organization of the paper.} 

The document is structured as follows: in the remainder of Section \ref{sec:stat main res} we state and comment the main results, whose proofs are contained in Section \ref{sec:proofs}. In Section \ref{sec:preliminaries} we collect, for reader's sake, all relevant results and bibliographical references on SP. Finally, in Appendix \ref{appendix} we prove the sharpness of a functional inequality introduced by the first-named author in \cite{Conforti17}, which plays an important role in this paper.


\subsection{First and second derivative of the entropic cost}

The regularity of the entropic cost w.r.t.\ to the time variable $T$ has never been investigated to the best of our knowledge, so that the following is the first result of such a kind and it plays a pivotal role in the study of both the long- and short-time behavior of $\cost(\mu,\nu)$.

\begin{theorem}\label{thm:derivcost}
Given \eqref{hyp pot} and \eqref{hyp marginals weak} the following hold:
\begin{itemize}
\item[(i)] the map $T \mapsto \cost(\mu,\nu)$ is $C^1((0,\infty))$, twice differentiable a.e.\ and the first derivative is given by
\[
\frac{\De}{\De T}\cost(\mu,\nu) = -\cons(\mu,\nu), \qquad \forall T > 0;
\]
\item[(ii)] the map $T \mapsto T\cost(\mu,\nu)$ is $C^1((0,\infty))$ and twice differentiable a.e. The first derivative is given for all $T>0$ by
\bes
\frac{\De}{\De T}\big(T\cost(\mu,\nu)\big) = \cost(\mu,\nu) - T\cons(\mu,\nu)
\ees
or equivalently by 
\bes
\frac{\De}{\De T}\big(T\cost(\mu,\nu)\big) =\frac{1}{2}\cH(\mu\,|\,\mm)+\frac{1}{2}\cH(\nu\,|\,\mm) + \frac{1}{4} \iint_0^T |\nabla\log\bar{\rho}^T_t|^2\, \bar{\rho}^T_t \De t\De\mm,
\ees
where $(\bar{\rho}^T,\bar{v}^T)$ is the optimal solution in \eqref{eq:bbs}. The second derivative writes as
\[
\begin{split}
\frac{\De^2}{\De T^2}\big(T\cost(\mu,\nu)\big) = \frac{1}{4}\frac{\De}{\De T}\iint_0^T|\nabla \log \bar{\rho}^T_t|^2\bar{\rho}^T_t\, \De t\De\mm = -2\cons(\mu,\nu) - T\frac{\De}{\De T}\cons(\mu,\nu)
\end{split}
\]
for a.e.\ $T>0$.
\end{itemize}
\end{theorem}

As implicitly stated above, by Theorem \ref{thm:derivcost} we see that $T \mapsto \cons(\mu,\nu)$ is continuous on $(0,\infty)$ and differentiable a.e.

\subsection{Long-time behavior of entropic cost and energy}

As a first step we show that

\begin{theorem}[Long-time behavior of entropic cost and energy]\label{thm:longtime}
Under \eqref{hyp pot} with $\kappa \geq 0$ and \eqref{hyp ref}, for any $\mu,\nu \in \cP(M)$ satisfying \eqref{hyp marginals weakest} it holds
\bes
    \lim_{T \to \infty} \cost(\mu,\nu) = \cH(\mu\,|\,\mm) + \cH(\nu\,|\,\mm).
\ees
If $\mu,\nu \in \cP(M)$ satisfy \eqref{hyp marginals weak}, then it also holds
\[
\lim_{T \to \infty}\cons(\mu,\nu) = 0.
\]
\end{theorem}

As concerns the long-time behavior of $\cost(\mu,\nu)$, we provide two different proofs:
\begin{itemize}
    \item the former is direct and relies on a $\Gamma$-convergence argument (see Section \ref{subsec:gamma});
    \item the latter is more technical and requires slightly stronger assumptions on $\mu$ and $\nu$, but as an advantage it allows us to determine the long-time behavior of the so-called ``$(f,g)$-decomposition'' of the optimal coupling in SP (see Section \ref{sec:preliminaries} for its definition, Section \ref{subsec:strong} for the proof).
\end{itemize}
As an application of this result we provide a new proof of the logarithmic Sobolev inequality based on entropic interpolations, in the same spirit of the recent paper \cite{GLRT19}, where an ``entropic'' proof of the HWI inequality is established.

\begin{cor}[Log-Sobolev inequality]\label{cor:logsob}
Under \eqref{hyp pot} with $\kappa > 0$, for any $\mu = \rho\mm \in \cP(M)$ it holds
\begin{equation}\label{eq:logsob}
    \ent(\mu\,|\,\mm) \leq \frac{1}{2\kappa}\int |\nabla\log\rho|^2\,\De\mu,
\end{equation}
where the right-hand side is set equal to $+\infty$ if $\log\rho$ is not locally Sobolev.
\end{cor}

As a further step, in the following result we improve Theorem \ref{thm:longtime} by providing sharp rates of convergence for both $\cost(\mu,\nu)$ and $\cons(\mu,\nu)$. The key message is that, under a positive curvature condition, the approximation error is asymptotically smaller than $\exp(-\kappa T/2)$, up to constant factors depending on $\cH(\mu\,|\,\mm)$ and $\cH(\nu\,|\,\mm)$, and the rate $\exp(-\kappa T/2)$ is sharp. As already pointed out, recall that if \eqref{hyp pot} holds with $\kappa>0$, then \eqref{hyp ref} also holds.

\begin{theorem}[Long-time behavior of entropic cost and energy - rate of convergence]\label{thm cost bound}
Let us assume that \eqref{hyp pot} with $\kappa>0$ and \eqref{hyp marginals weak} are satisfied. Then for all $T>0$ it holds
\be
\label{long time conv bound}
|\cost(\mu,\nu) - \cH(\mu\,|\,\mm) - \cH(\nu\,|\,\mm)| \leq \frac{2}{\exp(\kappa T/2)-1}\Big(\cH(\mu\,|\,\mm) + \cH(\nu\,|\,\mm)\Big)
\ee
and
\begin{equation}
\label{eq:long time energy bound}
|\cons(\mu,\nu)| \leq  \frac{\kappa \exp(-\kappa T/2)}{(1-\exp(-\kappa T/2))^2} \Big( \cH(\mu\,|\,\mm) + \cH(\nu\,|\,\mm)\Big)^2.
\end{equation}
Furthermore, the convergence rate $\exp(-\kappa T/2)$ in \eqref{long time conv bound} and \eqref{eq:long time energy bound} is sharp in the following sense: for all $\mu,\nu$ as in \eqref{hyp marginals weak} it holds
\[
\lim_{T \to \infty}\frac{1}{T}\log|\cost(\mu,\nu) - \cH(\mu\,|\,\mm) - \cH(\nu\,|\,\mm)| \leq -\frac{\kappa}{2}, \qquad \lim_{T \to \infty}\frac{1}{T}\log|\cons(\mu,\nu)| \leq -\frac{\kappa}{2}
\]
and there exists a triplet $(M',\sfd_g',\mm')$ satisfying \eqref{hyp pot} with $\kappa>0$ such that
\[
\lim_{T \to \infty}\frac{1}{T}\log|\cost(\mu,\nu) - \cH(\mu\,|\,\mm) - \cH(\nu\,|\,\mm)| \leq -\alpha\kappa, \qquad \lim_{T \to \infty}\frac{1}{T}\log|\cons(\mu,\nu)| \leq -\alpha\kappa
\]
holds for all $\mu,\nu$ satisfying \eqref{hyp marginals weak} if and only if $\alpha \leq 1/2$.
\end{theorem}

In other words, it may be possible to improve the constant factor in \eqref{long time conv bound}, but the convergence rate $\exp(-\kappa T/2)$ is asymptotically sharp.

\begin{remark}
{\rm
Actually, in Section \ref{subsec:cost bound} we shall prove the following (stronger) bounds
\begin{subequations}
\begin{align}
\label{eq:stronger bound 1}
& |\cost(\mu,\nu) - \cH(\mu\,|\,\mm) - \cH(\nu\,|\,\mm)| \leq 2\frac{\sqrt{\cH(\mu\,|\,\mm)\cH(\nu\,|\,\mm) + \delta\exp(-\kappa T/2)}}{\exp(\kappa T/2)-1}, \\
\label{eq:stronger bound 2}
& |\cons(\mu,\nu)| \leq  \frac{\kappa \exp(-\kappa T/2)}{(1-\exp(-\kappa T/2))^2} \Big( \cH(\mu\,|\,\mm)\cH(\nu\,|\,\mm) + \delta\exp(-\kappa T/2)\Big),
\end{align}
\end{subequations}
where
\[ 
\delta = \cH(\mu\,|\,\mm)^2 + \exp(-\kappa T/2) \cH(\mu\,|\,\mm) \cH(\nu\,|\,\mm) + \cH(\nu\,|\,\mm)^2.
\]
However we prefer the more compact, although slightly less precise, formulation given above. \fr
}
\end{remark}

\subsection{Short-time behavior of the entropic cost}

We provide a non-asymptotic bound for the difference $T\mathscr{C}_T(\mu,\nu)-W_2^2(\mu,\nu)/2$ under the assumption that the integral of the Fisher information along the displacement interpolation is finite and we also prove that this bound is sharp, as it coincides with the Taylor expansion of $T\mapsto T\cost(\mu,\nu)$ around $T=0$.\footnote{added in proof: With respect to a previous version of the paper, we have been able to remove quite demanding regularity assumptions. A similar but less general result (valid only in the Euclidean setting) has recently been obtained in \cite{chizat2020faster} by other means, hence our proof is independent.} This result is of particular interest as the second order term in the expansion of the rescaled cost has not been computed before (to the best of our knowledge) and its form is general enough to formulate more general version of \eqref{cost taylor} below, for instance in terms of $\Gamma$-convergence, that deserve to be the object of future work.

\begin{theorem}[Short-time behavior of the entropic cost]\label{thm short time}
Assume that \eqref{hyp pot} and \eqref{hyp marginals weak} hold. Then for all $T>0$ we have:
\be\label{short time bound}
0 \leq T\cost(\mu,\nu) -\frac{1}{2}W_2^2(\mu,\nu) \leq \frac{T}{2}\Big(\cH(\mu\,|\,\mm) + \cH(\nu\,|\,\mm)\Big) + \frac{T^2}{8} \iint_0^1 |\nabla \log \rho^0_t|^2 \rho^0_t\, \De t \De\mm,
\ee
where $(\rho^0_t \mm)_{t\in[0,1]}$ denotes the Wasserstein geodesic between $\mu$ and $\nu$ and $\int|\nabla\log\rho_t^0|^2\rho_t^0\,\De\mm := +\infty$ whenever $\rho_t^0$ is not Sobolev. 

If in addition the Bakry-\'Emery $\CD(\kappa,N)$ condition
\[
{\rm Ric}_{U,N} := {\rm Ric}_g - (N-n)\frac{{\rm Hess}(\exp(-\frac{2}{N-n}U))}{\exp(-\frac{2}{N-n}U)} \geq \kappa g
\]
holds for some $N \geq n = {\rm dim}(M)$ and $\iint_0^1 |\nabla \log \rho^0_t|^2 \rho^0_t\, \De t \De\mm < \infty$, then $T\mapsto T\cost(\mu,\nu)$ is twice differentiable at $T=0$ and the following expansion holds
\be
\label{cost taylor}
T\mathscr{C}_T(\mu,\nu) = \frac{1}{2}W_2^2(\mu,\nu) + \frac{T}{2}\Big(\cH(\mu\,|\,\mm) + \cH(\nu\,|\,\mm)\Big) + \frac{T^2}{8} \iint_0^1 |\nabla \log \rho^0_t|^2\rho^0_t \,\De t\De\mm + o(T^2).
\ee
\end{theorem}

Note, for instance, that the $\CD(\kappa,N)$ condition is satisfied when $\mm$ is the volume measure.

\subsection{Long time behavior of the mean field entropic cost}

In this section we move from the Riemannian to the Euclidean setting and turn our attention to the Mean Field Schr\"odinger problem (henceforth MFSP) recently introduced in \cite{backhoff2019mean}. Intuitively, this problem models the well-known thought experiment proposed by Schr\"odinger \cite{Schrodinger31,Schrodinger32} in the case when the underlying particles are not independent but interact through a pair potential $W$. An important difference with the classical case, is that MFSP cannot be reduced to a static problem. Therefore it is formulated for probability measures on the space of continuous paths $C([0,T];\RD)$, that we equip with the Borel $\sigma$-algebra generated by the norm $\|\cdot \|_{\infty}$. As usual, we also consider the canonical filtration on $C([0,T];\RD)$ and denote $(X_t)_{t\in[0,T]}$ the canonical process. MFSP is then stated as
\be
\label{nonlinearSP}\tag{MFSP}
\inf \Big\{\cH(\rmP\,|\,\Gamma(\rmP) ) \,:\,\,  \rmP \in \cP_{1}(C([0,T];\RD)), \,\rmP_0=\mu,\, \rmP_T=\nu\Big\}
\ee
where $\cP_1(C([0,T];\RD))$ denotes the space of probability measures over $C([0,T];\RD)$ with finite first moment w.r.t.\ $\|\cdot\|_\infty$ and, for any $\rmP \in \cP_1(C([0,T];\RD))$, $\rmP_0$ (resp.\ $\rmP_T$) denotes its marginal law at time $0$ (resp.\ $T$), whereas $\Gamma(\rmP)$ is the law of the controlled SDE
$$
\left\{ 
\begin{array}{rrl}
     \De Z_t&=& -\nabla W \ast\rmP_t (Z_t)\De t+\De B_t , \\
     Z_0&\sim& \mu.
\end{array}
\right .
$$
In the above we denoted by $\ast$ the usual convolution operator. The interaction potential $W$ satisfies the assumptions
\begin{equation}
\label{boundedhess}\tag{H'1}
\begin{split}
\text{$W$ is of class $C^{2}(\RD;\R)$ and symmetric, i.e. $W(x)=W(-x)$ for all $x \in \RD$}, \\
\sup_{z,v \in \mathbb{R}^d, |v|=1}  v\cdot \nabla^2W(z) \cdot v  < +\infty.
\end{split}
\end{equation}
The role of the functional $\cH(\cdot\,|\,\mm)$ is now taken by the functional $\cF:\cP_2(\RD)\rightarrow\R$, which is obtained by shifting by a constant the functional $\tilde{\cF}$ defined by
\be
\label{eq def F}
\tilde{\mathcal{F}}(\sigma) = \begin{cases} \displaystyle{\int_{\RD} \log\Big(\frac{\De \sigma}{\De\lambda}\Big)\, \De\sigma + \int_{\RD} W \ast \frac{\De\sigma}{\De\lambda}\, \De\sigma}, \quad & \mbox{if $\sigma \ll \leb $}\\
+\infty, \quad & \mbox{otherwise}
\end{cases}
\ee
where $\leb$ denotes the Lebesgue measure on $\RD$, in such a way that 
\[
\inf_{\sigma \in\cP_2(\RD) : \int x \sigma(\De x) = \int x \mu(\De x)} \cF(\sigma) = 0.
\]
In \cite{backhoff2019mean} the mean field entropic cost is defined as the optimal value in \eqref{nonlinearSP}. However, if $W=0$ this definition does not give back the usual entropic cost $\cost$ and the reason is simply the following: in the large deviations formulation of \eqref{cost def} particles are sampled with initial distribution the invariant measure $\mm$, whereas in MFSP particles are sampled according to $\mu$. For this reason, and to strengthen the analogy with \eqref{cost def}, we prefer to define the mean field entropic cost $\mfcost$ as
\be
\label{eq:mfcost}
\mfcost(\mu,\nu) := \cF(\mu)+ \cH(\rmP^* | \Gamma(\rmP^*)), \quad  \rmP^* \textrm{ optimal in \eqref{nonlinearSP}.} 
\ee
With this definition we recover $\cost(\mu,\nu)$ defined in \eqref{cost def} when $W=0$.

After this premise, the following plays the same role of assumption \eqref{hyp pot}
\be\label{W convexity ass}\tag{H'2}
\exists \kappa > 0 \quad \textrm{s.t.} \quad  \forall z\in\RD, \quad 2\nabla^2 W(z) \geq \kappa \mathbb{I}_{d\times d},
\ee
while regarding the marginal constraints $\mu$ and $\nu$, we assume that they have finite second moment, they belong to the domain of $\cF$ and have the same barycenter, i.e.\
\be
\label{same mean ass}\tag{H'3} 
\mu,\nu \in \cP_2(\RD),\quad \cF(\mu),\cF(\nu) < +\infty,\quad \int_{\RD} x\, \mu (\De x) =  \int_{\RD} x\, \nu (\De x) .
\ee
The proof we gave of Theorem \ref{thm:derivcost} is hard to replicate for MFSP essentially because uniqueness of optimizers is currently not known. However, the long-time behavior of the mean field entropic cost can still be studied and exponential rate of convergence can be derived as well.
 
\begin{theorem}\label{thm meanfieldlongtime}
Assume \eqref{boundedhess}-\eqref{same mean ass} and that $\cF(\mu),\cF(\nu)<+\infty$. Then:
\bei 
\item[(i)] it holds
 \bes
\lim_{T \to \infty} \mfcost(\mu,\nu) = \cF(\mu)+\cF(\nu);
 \ees
\item[(ii)] the rate of convergence is at least $\exp(-\kappa T)$, i.e.\ there exists a decreasing function $B(\cdot)$  such that
\be\label{exp entropy bound} |\mfcost(\mu,\nu)-\cF(\mu)-\cF(\nu)|\leq B(\kappa)(\cF(\mu) +\cF(\mu)) \exp(-\kappa  T/2) 
\ee
uniformly in $T \geq 1$. 
\eei
\end{theorem}

\begin{remark}
{\rm
One could be more precise in the above statement and get that 
\beas
(\cF(\mu)+\cF(\nu))\frac{\sinh(\kappa T)}{\sinh(\kappa T/2)} +  2\frac{\exp(\kappa T/2)}{(\exp(\kappa T/2) -1)^2}\Big(\cF(\mu)\cF(\nu) + \delta \exp(-\kappa T/2) \Big)^{1/2} \\
\leq \mfcost(\mu,\nu) - \cF(\mu)+\cF(\nu) \leq \frac{1}{\exp(\kappa T/2)-1 }(\cF(\mu)+\cF(\nu))
\eeas
where $\delta = \cF(\mu)^2 + \cF(\nu)^2 + \exp(-\kappa T/2) \cF(\mu)\cF(\nu)$. For sake of clarity, we prefer the more compact, although slightly less precise, formulation given above.\fr
}
\end{remark}

\begin{remark}\label{rem:factor2}
{\rm 
It is important to stress that in \cite{backhoff2019mean} the interaction potential $W$ satisfies $\nabla^2 W(z) \geq \kappa \mathbb{I}_{d\times d}$ instead of \eqref{W convexity ass}. We prefer to work with the latter condition rather than with the former because in such a way the convergence rate \eqref{exp entropy bound} is consistent with \eqref{long time conv bound}. As a flip side, when addressed to \cite{backhoff2019mean} the reader has to keep in mind this difference.\fr
}
\end{remark}

\section{Preliminaries}\label{sec:preliminaries}

In this section we collect some useful results concerning Markov semigroups and Schr\"odinger problem, either already present in the literature or extended to our framework.

\paragraph{Markov semigroups and heat kernels.} It is well known that there is a Markov semigroup $\sfP_t$ associated to the SDE \eqref{SDE}, which has $\sfL = \Delta/2 - \nabla U\cdot\nabla$ as generator and $\mm(\De x) = \exp(-2U(x))\mathrm{vol}(\De x)$ as invariant measure. We address the reader to \cite{BakryGentilLedoux14} for a comprehensive discussion on the topic. However, it is worth stressing that in \cite{BakryGentilLedoux14} a different convention is adopted and the generator of $\sfP_t$ is $2\sfL = \Delta - 2\nabla U\cdot \nabla$; as a consequence, there exists a discrepancy in the factors appearing in the estimates \eqref{eq:apriori}, \eqref{eq:bakry-emery}, \eqref{eq:lip reg}, \eqref{eq:hamilton}, \eqref{eq:upperbound}, \eqref{eq:lowerbound} below and the corresponding estimates in \cite{BakryGentilLedoux14}.

With this remark in mind, let us present all the useful information about $\sfP_t$. First of all, it enjoys the following standard a priori estimate
\begin{equation}
\label{eq:apriori}
\|\sfL\sfP_t f\|_{L^2(\mm)} \leq \frac{2}{t^2}\|f\|_{L^2(\mm)}, \qquad \forall t>0,\,\forall f \in L^2(\mm),
\end{equation}
which can be obtained by differentiating $t \mapsto \||\nabla\sfP_t f|\|^2_{L^2(\mm)}$. The semigroup is also ergodic. This means that:
\begin{itemize}
    \item[(i)] if $\mm(M) = 1$, then for all $f \in L^2(\mm)$ it holds
    \begin{equation}\label{eq:ergodicity}
        \lim_{t \to \infty}\sfP_t f = \int f\,\De\mm, \qquad \text{in } L^2(\mm),
    \end{equation}
    \item[(ii)] if $\mm(M) = \infty$, then $\sfP_t f \to 0$ in $L^2(\mm)$ for all $f \in L^2(\mm)$.
\end{itemize}
The curvature assumption \eqref{hyp pot} then yields many important consequences, the first of which is the Bakry-\'Emery commutation estimate
\begin{equation}\label{eq:bakry-emery}
    |\nabla \sfP_t u|^2 \leq e^{-\kappa t}\sfP_t(|\nabla u|^2), \qquad \forall u \in C^\infty_c(M),\, \forall t \geq 0.
\end{equation}
For its proof as well as for all the regularizing properties of $\sfP_t$ that will be used throughout the paper, we address once more the reader to \cite{BakryGentilLedoux14}. Secondly, $\sfP_t$ enjoys an $L^\infty$-Lipschitz regularization (see \cite{Bakry06}), namely for all $u \in L^\infty(\mm)$ it holds
\begin{equation}\label{eq:lip reg}
    \|\nabla \sfP_t u\|_{L^\infty(\mm)} \leq \frac{1}{\sqrt{I_\kappa(t)}}\|u\|_{L^\infty(\mm)} \quad \text{with} \quad I_\kappa(t) := \frac{\exp(\kappa t)-1}{\kappa} , \qquad \forall t > 0.
\end{equation}
Then let us recall that under \eqref{hyp pot} Hamilton's gradient estimate is satisfied (see \cite{Kotschwar07}): for any positive function $u \in L^p \cap L^\infty(\mm)$ for some $p \in [1,\infty)$ it holds
\begin{equation}
\label{eq:hamilton}
t|\nabla\log\sfP_t u|^2 \leq 2(1+\kappa^- t)\log\Big(\frac{\|u\|_{L^\infty(\mm)}}{\sfP_t u}\Big), \qquad \forall t>0
\end{equation}
pointwise, where $\kappa^- := \max\{0,-\kappa\}$. Within our framework it is also well known (see for instance \cite{Grigoryan09}) that there exists a unique kernel (transition probability) $\hp_t(x,y)$ representing $\sfP_t$ in the following sense:
\begin{equation}\label{eq:heat representation}
\sfP_t u(x) = \int u(y)\hp_t(x,y)\mm(\De y), \qquad \forall u \in L^\infty(\mm).
\end{equation}
The function $\hp_t$ can also be seen as the density of $R_{0,T}$ (the joint law at time 0 and $T$ of the solution to \eqref{SDE}) w.r.t.\ $\mm \otimes \mm$; it is a smooth function on $(0,\infty) \times M \times M$ and the second-named author recently proved in \cite{Tamanini19} that in great generality (and in particular under the assumption \eqref{hyp pot}) the following upper Gaussian estimate for the kernel holds
\begin{equation}\label{eq:upperbound}
\hp_t(x,y) \leq \frac{1}{\sqrt{\mm(B_{\sqrt{t/2}}(x))\mm(B_{\sqrt{t/2}}(y))}}\exp\Big(C_\varepsilon(1+C_\kappa t) - \frac{\sfd^2(x,y)}{(2+\varepsilon)t}\Big)
\end{equation}
for all $t>0$, $x,y \in M$ and $\varepsilon > 0$, where $C_\kappa$ can be chosen equal to 0 if $\kappa \geq 0$ in \eqref{hyp pot}. As concerns Gaussian lower bounds, if \eqref{hyp ref} holds then by \cite[Corollary 1.3]{Wang11} the following is satisfied
\begin{equation}\label{eq:lowerbound}
\hp_t(x,y) \geq \exp\Big(-\frac{\kappa\sfd^2(x,y)}{2(e^{\kappa t/2}-1)}\Big)
\end{equation}
for all $t>0$ and $x,y \in M$.

\paragraph{The relative entropy functional.} Let us first recall the definition of the relative entropy functional in the case of a reference measure with possibly infinite mass (see \cite{Leonard14b} for more details). Given a $\sigma$-finite measure $q$ on $M$, there exists a measurable function $W : M \to [0,\infty)$ such that
\[
z_W := \int e^{-W}\De q < +\infty.
\]
Introducing the probability measure $q_W := z_W^{-1}e^{-W}q$, for any $p \in \cP(M)$ such that $\int W\De p < +\infty$ the Boltzmann-Shannon entropy is defined as
\[
\cH(p\,|\,q) := \cH(p\,|\,q_W) - \int W\De p - \log z_W
\]
where $\cH(p\,|\,q_W)$ is in turn defined as in \eqref{eq:entropy}. Notice that Jensen's inequality and the fact that $q_W \in \cP(M)$ grant that $\cH(p\,|\,q_W)$ is well defined and non-negative, in particular the definition makes sense. The definition is also meaningful, because if $\int W'\De p < +\infty$ for another function $W'$ such that $z_{W'} < +\infty$, then
\[
\cH(p\,|\,q_W) - \int W\De p - \log z_W = \cH(p\,|\,q_{W'}) - \int W'\De p - \log z_{W'}.
\]
Hence $\cH(p\,|\,q)$ is well defined for all $p \in \cP(M)$ such that $\int W\De p < +\infty$ for some non-negative measurable function $W$ with $z_W < +\infty$. In particular, it is well known that under \eqref{hyp pot} $W$ can be chosen as $C\sfd^2(\bar{x},\cdot)$ for any $\bar{x} \in M$ and $C>0$ sufficiently large (see for instance \cite[Theorem 4.26]{Sturm06I}).

\paragraph{The $(f,g)$-decomposition.} It is important to stress that solving \eqref{cost def} is equivalent to finding non-negative Borel functions $f^T,g^T$, also called decomposition, such that
\begin{equation}\label{eq:schr system}
\mu = f^T \sfP_T g^T\mm, \qquad \nu = g^T \sfP_T f^T \mm.
\end{equation}
This pair of equations is known as Schr\"odinger system and its solvability holds under very mild assumptions (see \cite{Leonard14}), in particular under:
\begin{itemize}
\item \eqref{hyp marginals weak}, as a consequence of \cite[Proposition 2.1]{GigTam18}, the smoothness and positivity of the heat kernel and the boundedness of $\supp(\mu),\supp(\nu)$;
\item \eqref{hyp pot},  \eqref{hyp ref} and \eqref{hyp marginals weakest}, because of \cite[Proposition 2.5]{Leonard14} and \eqref{eq:lowerbound}.
\end{itemize}
Still under \eqref{hyp pot}-\eqref{hyp marginals weak}, the couple $(f^T,g^T)$ solving \eqref{eq:schr system} is unique up to the trivial transformation $(f^T,g^T) \mapsto (cf^T,g^T/c)$ with $c>0$, as proven for instance in \cite[Proposition 2.1]{GigTam18}: this fact will play an important role in several proofs (e.g.\ the extension of \eqref{eq:k-convexity} to our setting and Lemma \ref{lem:fg unif conv}). A good feature of $f^T,g^T$ is that they inherit the regularity (smoothness and integrability) of $\rho,\sigma$, the densities of $\mu,\nu$ respectively. More precisely, 
\begin{itemize}
    \item[(a)] $f^T,g^T \in L^\infty(\mm)$ since so are $\rho,\sigma$;
    \item[(b)] if $\rho \in C^k(M)$ for some $k \in \mathbb{N} \cup \{\infty\}$ (resp.\ $\sigma$), then $f^T$ (resp.\ $g^T$) also belongs to $C^k(M)$.
\end{itemize}
A proof of this property can be found in \cite[Proposition 2.7]{GLRT19}.

About point (a) a more quantitative statement is actually possible. In \cite[Proposition 2.1]{GigTam18} the second-named author proved that for all $\mu,\nu$ as in \eqref{hyp marginals weak}, the following integral bounds hold for the decomposition $(f^T,g^T)$
\[
\|f^T\|_{L^\infty(\mm)}\|g^T\|_{L^1(\mm)} \leq \frac{\|\rho\|_{L^\infty(\mm)}}{c_T}, \qquad \|g^T\|_{L^\infty(\mm)}\|f^T\|_{L^1(\mm)} \leq \frac{\|\sigma\|_{L^\infty(\mm)}}{c_T},
\]
where $c_T$ is a suitable positive constant. If we normalize $g^T$ in such a way that 
\begin{equation}\label{eq:gauge}
    \|g^T\|_{L^1(\mm)} = 1, \qquad \forall T > 0,
\end{equation}
and from now on this choice will always be done, then the bounds above become
\[
\|f^T\|_{L^\infty(\mm)} \leq \frac{\|\rho\|_{L^\infty(\mm)}}{c_T}, \qquad \|g^T\|_{L^\infty(\mm)}\|f^T\|_{L^1(\mm)} \leq \frac{\|\sigma\|_{L^\infty(\mm)}}{c_T}
\]
Aim of the next lemma is to improve this result by showing that the same kind of bounds holds when $T$ ranges in a compact subset of $(0,\infty)$ or even closed half-line if we also assume \eqref{hyp ref}.

\begin{lemma}\label{lem:fg bounds}
Given \eqref{hyp pot} and $\mu,\nu$ as in \eqref{hyp marginals weak}, the following hold:
\begin{itemize}
\item[(i)] for all $0<T_0<T_1<\infty$ there exists $c_{T_0,T_1}>0$ such that
\begin{equation}\label{eq:fg local bounds}
\|f^T\|_{L^\infty(\mm)} \leq \frac{\|\rho\|_{L^\infty(\mm)}}{c_{T_0,T_1}}, \quad \|g^T\|_{L^\infty(\mm)}\|f^T\|_{L^1(\mm)} \leq \frac{\|\sigma\|_{L^\infty(\mm)}}{c_{T_0,T_1}}, \qquad \forall T \in [T_0,T_1].
\end{equation}
\item[(ii)] if \eqref{hyp ref} is satisfied, then for all $T_0>0$ there exists $c_{T_0}>0$ such that
\begin{equation}\label{eq:fg infinity bounds}
\|f^T\|_{L^\infty(\mm)} \leq \frac{\|\rho\|_{L^\infty(\mm)}}{c_{T_0}}, \quad \|g^T\|_{L^\infty(\mm)}\|f^T\|_{L^1(\mm)} \leq \frac{\|\sigma\|_{L^\infty(\mm)}}{c_{T_0}}, \qquad \forall T \geq T_0.
\end{equation}
\end{itemize}
\end{lemma}

\begin{proof}
The first equation in the Schr\"odinger system \eqref{eq:schr system} and the representation formula \eqref{eq:heat representation} entail
\[
f^T(x)\int g^T(y) \hp_T(x,y)\De\mm(y) = \rho(x), \qquad \textrm{for }\mm\textrm{-a.e. }x.
\]
As $\hp_T(x,y)$ is smooth in $T \in (0,\infty)$ and $x,y \in M$ and $f^T,g^T$ are uniformly compactly supported (since they are supported in $\supp(\mu)$ and $\supp(\nu)$ respectively) we deduce that there exists $c_{T_0,T_1}>0$ such that $\hp_T(x,y) \geq c_{T_0,T_1}$ in $\supp(\mu) \times \supp(\nu)$ for all $T \in [T_0,T_1]$, whence
\[
f^T = f^T\|g^T\|_{L^1(\mm)} \leq \frac{\|\rho\|_{L^\infty(\mm)}}{c_{T_0,T_1}}, \qquad \mm\textrm{-a.e. in } \supp(\mu)
\]
provided $T \in [T_0,T_1]$, and this proves the first inequality in \eqref{eq:fg local bounds}. For the second one it is sufficient to swap the roles of $f^T$ and $g^T$. 
If we further assume \eqref{hyp ref}, then the Gaussian lower bound \eqref{eq:lowerbound} holds. Hence for all $T_0>0$ there exists $c_{T_0}>0$ such that $\hp_T(x,y) \geq c_{T_0}$ in $\supp(\mu) \times \supp(\nu)$ for all $T \geq T_0$, whence by the same argument as above the first inequality in \eqref{eq:fg infinity bounds} follows. By interchanging the roles of $f^T$ and $g^T$, the same conclusion follows for $g^T$.
\end{proof}

\paragraph{The Benamou-Brenier formulation.} The equivalence between \eqref{eq:schr system} and \eqref{cost def} is extremely fruitful, as it enables to fully describe the optimal pair $(\bar{\rho}^T,\bar{v}^T)$ in the fluid-dynamical description \eqref{eq:bbs} of the entropic cost. Indeed,
\[
\bar{\rho}_t^T = \sfP_t f^T \sfP_{T-t}g^T, \quad \bar{v}_t^T = \frac{1}{2}\nabla\log \sfP_{T-t}g^T - \frac{1}{2}\nabla\log \sfP_t f^T, \qquad t \in [0,T].
\]
Actually, \eqref{eq:bbs} and the identities above are nothing but a reparametrization of the Benamou-Brenier-like formula for the entropic cost established in \cite[Theorem 4.2]{GigTam19} and \cite{Tamanini20}, which reads as
\begin{equation}\label{eq:bbs2}
T\cost(\mu,\nu) = \frac{T}{2}\Big(\ent(\mu\,|\,\mm) + \ent(\nu\,|\,\mm)\Big) + \inf_{(\rho,v)}\iint_0^1 \Big(\frac12|v_t|^2 + \frac{T^2}{8}|\nabla\log\rho_t|^2\Big) \rho_t\De t\De\mm,
\end{equation}
and of
\[
\rho_t^T := \sfP_{Tt}f^T \sfP_{T(1-t)}g^T, \quad v_t^T := \frac{T}{2}\nabla\log \sfP_{T(1-t)}g^T - \frac{T}{2}\nabla\log \sfP_{Tt}f^T, \qquad t \in [0,1].
\]
In other words
\begin{equation}\label{eq:reparametrization}
    \rho_t^T = \bar{\rho}_{Tt}^T \quad \text{and} \quad v_t^T = T\bar{v}_{Tt}^T.
\end{equation}
In \eqref{eq:bbs2} the infimum runs over all weak solutions of the continuity equation \eqref{eq:conteq} satisfying the marginal constraints $\rho_0\mm = \mu$ and $\rho_1\mm = \nu$, while $(\rho^T,v^T)$ defined above is the unique optimal density-velocity couple for \eqref{eq:bbs2}. It is also useful to see the second term in the right-hand side of \eqref{eq:bbs2} as an action functional, whose arguments are the curves $(\rho_t)$ and $(v_t)$, i.e.
\[
\mathscr{A}_T(\rho,v) := \iint_0^1 \Big(\frac12|v_t|^2 + \frac{T^2}{8}|\nabla\log\rho_t|^2\Big) \rho_t\De t\De\mm.
\]
When $T=0$, $\mathscr{A}_T$ is a purely kinetic energy and by the well-known Benamou-Brenier formulation of the optimal transport problem \cite{BenamouBrenier00}, 
\begin{equation}
\label{eq:benamou-brenier}
\frac{1}{2}W_2^2(\mu,\nu) = \inf_{(\rho,v)} \mathscr{A}_0(\rho,v),
\end{equation}
where the infimum is taken over the same set as in \eqref{eq:bbs2}. In this case optimal couples shall be denoted by $(\rho^0,v^0)$. Let us also recall that
\begin{equation}\label{eq:cost representations}
\begin{split}
\cost(\mu,\nu) & = \ent(\mu\,|\,\mm) + \frac{T}{2} \iint_0^1 |\nabla\log \sfP_{T(1-t)}g^T|^2\rho_t^T\,\De t\De\mm \\
& = \ent(\nu\,|\,\mm) + \frac{T}{2}\iint_0^1 |\nabla\log \sfP_{Tt} f^T|^2\rho_t^T\,\De t\De\mm,
\end{split}
\end{equation}
which trivially implies
\begin{equation}\label{eq:finiteness}
    \iint_0^1 |\nabla\log \sfP_{Tt} f^T|^2\rho_t^T\,\De t\De\mm, \, \iint_0^1 |\nabla\log \sfP_{T(1-t)}g^T|^2\rho_t^T\,\De t\De\mm < \infty
\end{equation}
Finally, the expression of the conserved quantity $\cons(\mu,\nu)$ in terms of the rescaled optimal couple $(\rho^T,v^T)$ reads as
\begin{equation}
\label{eq:energy representation}
\begin{split}
\cons(\mu,\nu) & = \frac{1}{2T^2} \int |v^T_t|^2 \rho^T_t \De \mm - \frac18 \int |\nabla \log \rho^T_t|^2  \rho^T_t \De \mm \\
& = -\frac{1}{2}\int \nabla\log \sfP_{T(1-t)}g^T \cdot \nabla\log \sfP_{Tt} f^T\,\rho_t^T\,\De\mm.
\end{split}
\end{equation}

\paragraph{Dual formulation.} For future reference, it is also worth mentioning the fact that, for all $\mu,\nu$ satisfying \eqref{hyp marginals weak}, the (rescaled) entropic cost admits the following dual representation (see \cite{GigTam19,Tamanini20})
\begin{equation}
\label{eq:dual representation}
T\cost(\mu,\nu) = T\cH(\mu\,|\,\mm) + \sup_{\phi \in C_b(M)}\bigg\{\int \sfQ_1^T\phi\,\De\mu - \int\phi\,\De\nu \bigg\},
\end{equation}
where $Q_1^T\phi := -T\log\sfP_T(\exp(-\phi/T))$.

\paragraph{Geometric information.} As regards the relationship between Schr\"odinger problem and lower Ricci bounds, let us recall that in \cite{Conforti17} the first-named author showed that \eqref{hyp pot} implies the following distorted $\kappa$-convexity inequality of the entropy along entropic interpolations:
\begin{equation}\label{eq:k-convexity}
\begin{split}
   \ent(\mu_t^T\,|\,\mm) & \leq \frac{1-\exp(-\kappa T(1-t))}{1-\exp(-\kappa T)}\ent(\mu\,|\,\mm) + \frac{1-\exp(-\kappa Tt)}{1-\exp(-\kappa T)}\ent(\nu\,|\,\mm) \\
   & \qquad - \frac{\cosh(\kappa T/2) - \cosh(\kappa T(t-1/2))}{\sinh(\kappa T/2)}\cost(\mu,\nu),
\end{split}
\end{equation}
for all $t \in [0,1]$, where $\mu_t^T := \rho_t^T\mm$. Truth to be told, the result in \cite{Conforti17} was stated in the framework of compact Riemannian manifolds satisfying \eqref{hyp pot} with $\kappa>0$ and endowed with the volume measure, but the proof can be adapted to our setting when \eqref{hyp ref} holds, by relying on \cite[Lemma 3.7]{GLRT19}. Indeed, if we set
\begin{equation}\label{eq:delta definitions}
\begin{split}
\rho_t^{T,\delta} := c_\delta(\sfP_{Tt}f^T+\delta)(\sfP_{T(1-t)}g^T + \delta), \qquad \mu_t^{T,\delta} := \rho_t^{T,\delta}\mm, \\ v_t^{T,\delta} := \frac{T}{2}\nabla\log(\sfP_{T(1-t)}g^T + \delta) - \frac{T}{2}\nabla\log(\sfP_{Tt}f^T+\delta),
\end{split}
\end{equation}
where $c_\delta$ is a normalization constant so that $\rho_t^{T,\delta}$ is still a probability density, then the identities
\[
\begin{split}
& \partial_t\log(\sfP_{Tt}f^T + \delta) = \frac{T}{2}|\nabla\log(\sfP_{Tt}f^T + \delta)|^2 + T\sfL\log(\sfP_{Tt}f^T + \delta) \\
& \partial_t\log(\sfP_{T(1-t)}g^T + \delta) = \frac{T}{2}|\nabla\log(\sfP_{T(1-t)}g^T + \delta)|^2 + T\sfL\log(\sfP_{T(1-t)}g^T + \delta) \\
& \partial_t\rho_t^{T,\delta} + {\rm div}_\mm(v_t^{T,\delta}\rho_t^{T,\delta}) = 0
\end{split}
\]
hold both in the classical sense and as strong $W^{1,2}$-limits, where $\sfL := \Delta/2 - \nabla U \cdot \nabla$ is the generator of $(\sfP_t)$. Together with \eqref{eq:lip reg} and \eqref{eq:hamilton}, this is sufficient to follow the lines of \cite[Lemma 3.6, Lemma 3.7]{Conforti17} and \cite[Lemma 3.7]{GLRT19} and deduce the following

\begin{lemma}
\label{lem:fg derivatives}
Under \eqref{hyp pot}, given $\mu,\nu$ as in \eqref{hyp marginals weak} and with the same notations as in \eqref{eq:delta definitions}, for all $\delta > 0$ and $t \in [0,1]$ define
\begin{equation}\label{eq:forward backward}
h_f(t) := \int\log(\sfP_{Tt}f^T + \delta)\,\De\mu_t^{T,\delta}, \qquad h_b(t) := \int\log(\sfP_{T(1-t)}g^T + \delta)\,\De\mu_t^{T,\delta}.
\end{equation}
Then $h_f \in C([0,1]) \cap C^2((0,1])$, $h_b \in C([0,1]) \cap C^2([0,1))$,
\begin{equation}
\label{eq:first derivatives}
h_f'(t) = -\frac{T}{2}\int |\nabla\log(\sfP_{Tt}f^T + \delta)|^2\,\De\mu_t^{T,\delta}, \qquad h_b'(t) = \frac{T}{2}\int |\nabla\log(\sfP_{T(1-t)}g^T + \delta)|^2\,\De\mu_t^{T,\delta}
\end{equation}
and
\[
h_f''(t) \geq -\kappa T h_f'(t), \qquad h_b''(t) \geq \kappa T h_b'(t).
\]
\end{lemma}

By \cite[Lemma 4.1]{Conforti17} and following the proof of Theorem 1.4 therein, we obtain
\[
\begin{split}
   \ent(\mu_t^{T,\delta}\,|\,\mm) & \leq \frac{1-\exp(-\kappa T(1-t))}{1-\exp(-\kappa T)}\ent(\mu_0^{T,\delta}\,|\,\mm) + \frac{1-\exp(-\kappa Tt)}{1-\exp(-\kappa T)}\ent(\mu_1^{T,\delta}\,|\,\mm) \\
   & \qquad - \frac{\cosh(\kappa T/2) - \cosh(\kappa T(t-1/2))}{\sinh(\kappa T/2)}\cost(\mu_0^{T,\delta},\mu_1^{T,\delta}).
\end{split}
\]
It is now sufficient to pass to the limit as $\delta \downarrow 0$: $\ent(\mu_t^{T,\delta}\,|\,\mm)$ converges to $\ent(\mu_t^T\,|\,\mm)$, $\ent(\mu_0^{T,\delta}\,|\,\mm)$ to $\ent(\mu\,|\,\mm)$ and $\ent(\mu_1^{T,\delta}\,|\,\mm)$ to $\ent(\nu\,|\,\mm)$ by dominated convergence theorem; as concerns $\cost(\mu_0^{T,\delta},\mu_1^{T,\delta})$, this converges to $\cost(\mu,\nu)$ for the following reason: $\tilde{f}^T := \sqrt{c_\delta}(f^T+\delta)$ and $\tilde{g}^T := \sqrt{c_\delta}(g^T+\delta)$ solve \eqref{eq:schr system} with $\mu_0^{T,\delta}$, $\mu_1^{T,\delta}$ as marginal constraints, so that
\[
\cost(\mu_0^{T,\delta},\mu_1^{T,\delta}) = \ent(\tilde{f}^T \otimes \tilde{g}^T R_{0,T}\,|\,R_{0,T})
\]
and again by dominated convergence the right-hand side above converges to 
\[
\ent(f^T \otimes g^T R_{0,T}\,|\,R_{0,T}) = \cost(\mu,\nu).
\]
Let us stress that the fact that $\tilde{f}^T$ and $\tilde{g}^T$ solve \eqref{eq:schr system} with $\mu_0^{T,\delta}$, $\mu_1^{T,\delta}$ as marginal constraints and the equivalence between SP and \eqref{eq:schr system} immediately imply that $(\rho^{T,\delta},v^{T,\delta})$ is optimal in \eqref{eq:bbs2}.

\section{Proof of the main results}\label{sec:proofs}

\subsection{A \texorpdfstring{$\Gamma$}{Gamma}-convergence proof of Theorem \ref{thm:longtime}}\label{subsec:gamma}

A rather straightforward proof of the long-time behavior of $\cost(\mu,\nu)$ can be obtained by a $\Gamma$-convergence argument. Its essence is contained in the following

\begin{lemma}\label{lem:gamma convergence}
Under assumption \eqref{hyp pot} with $\kappa \geq 0$ and \eqref{hyp ref}, assume that $\mu,\nu \in \cP_2(M)$ with $\cH(\mu\,|\,\mm),\cH(\nu\,|\,\mm) < \infty$. Let $(T_n)_{n\in\N} \subset \R_{\geq 0}$ be a sequence converging to $+\infty$ and, for any $n$, consider the functional $\ent(\cdot\,|\,R_{0,T_n})$ defined on $\Pi(\mu,\nu)$ endowed with the weak topology. Then \bes \Gamma-\lim_{n\to\infty}\ent(\cdot\,|\,R_{0,T_n}) = \ent(\cdot\,|\,\mm\otimes\mm)\ees
\end{lemma}

\begin{proof}
As a first step, we claim that for any lower bounded lower semicontinuous function $\phi$ on $M \times M$ it holds
\begin{equation}
\label{eq:portmanteau}
\liminf_{n \to \infty}\int_{M \times M} \phi\,\De R_{0,T_n} \geq \int_{M \times M}\phi\,\De(\mm \otimes \mm).
\end{equation}
Since $R_{0,T_n}(M \times M) = 1$, the claim is trivial for constant functions; hence, without loss of generality we can assume that $\phi \geq 0$ and, under this further assumption, the Gaussian lower bound \eqref{eq:lowerbound} together with Fatou's lemma yields \eqref{eq:portmanteau}. By Portmanteau theorem this is equivalent to say that
\be
\label{eq:weakconv} 
R_{0,T_n} \rightharpoonup \mm\otimes\mm
\ee 
in duality with $C_b(M \times M)$.

After this premise, let $\pi \in \Pi(\mu,\nu)$, $(\pi_n) \subset \Pi(\mu,\nu)$ and assume that $\pi_n \rightharpoonup \pi$. The lower semicontinuity of the relative entropy w.r.t.\ to both its arguments together with \eqref{eq:weakconv} gives that 
\bes
\liminf_{n\to\infty} \ent(\pi_n\,|\,R_{0,T_n}) \geq \ent(\pi\,|\,\mm\otimes\mm),
\ees
thus the $\Gamma$-liminf inequality. To complete the proof, let $\pi \in \Pi(\mu,\nu)$ and find a sequence $(\pi^n) \subset \Pi(\mu,\nu)$ such that
\begin{equation}
\label{eq:gammalimsup}
\limsup_{n \to \infty} \ent(\pi^n\,|\,R_{0,T_n}) \leq \ent(\pi\,|\,\mm\otimes\mm).
\end{equation}
To this aim it is not restrictive to assume that $\cH(\pi\,|\,\mm\otimes\mm) < \infty$ and it is also easy to see that $\pi^n \equiv \pi$ satisfies \eqref{eq:gammalimsup}. Indeed, the Gaussian lower bound \eqref{eq:lowerbound} and the fact that \eqref{hyp pot} holds with $\kappa \geq 0$ imply that 
\bes
-\log \hp_{T_n}(x,y)\leq \frac{1}{2T_n} \sfd^2(x,y)
\ees
for all $x,y \in M$ and for all $n \in \N$. In particular, since $\sfd^2(x,y) \leq 2\sfd^2(x,z) + 2\sfd^2(y,z)$ for all $z \in M$ and $\mu,\nu \in \cP_2(M)$ we have that $\int_{M \times M} -\log \hp_{T_n}(x,y)\pi(\De x \De y)$ is well defined for all $n$ and
\be
\label{eq:gammaconv1}
-\int_{M \times M} \log \hp_{T_n}(x,y)\pi(\De x \De y)\leq \frac{C}{2T_n}
\ee
for some constant $C$ independent of $n$. Next, observe that by definition of $\hp_{T_n}(\cdot,\cdot)$ and $\pi^n$ we have
\be
\label{eq:gammaconv2}
\ent(\pi^n\,|\,R_{0,T_n}) = \ent(\pi\,|\,\mm\otimes\mm) -\int \log \hp_{T_n}(x,y) \pi(\De x,\De y),
\ee
where the sum on the right-hand side is well defined because $0 \leq \ent(\pi\,|\,\mm\otimes\mm) < \infty$ by assumption and the fact that $\mm\otimes\mm$ is a probability, and a fortiori $\int \log \hp_{T_n}\, \De \pi \in [-\infty,+\infty)$, since $\ent(\pi^n\,|\,R_{0,T_n}) \geq 0$. Plugging \eqref{eq:gammaconv1} into \eqref{eq:gammaconv2} and letting $n \to \infty$ yields \eqref{eq:gammalimsup} and thus the $\Gamma$-limsup inequality.
\end{proof}

Let us point out that in the previous lemma and thus in Theorem \ref{thm:longtime}, \eqref{hyp ref} is required only when $\kappa=0$. As regards the long-time behavior of $\cons(\mu,\nu)$, we need to determine in which way $\cons(\mu,\nu)$ is controlled in terms of $\cost(\mu,\nu)$.

\begin{lemma}\label{lem:cost bounds cons}
Under \eqref{hyp pot} and \eqref{hyp marginals weak} it holds
\begin{equation}
\label{eq:cost bounds cons}
T|\cons(\mu,\nu)| \leq \cost(\mu,\nu) - \frac12 \Big(\ent(\mu\,|\,\mm) + \ent(\nu\,|\,\mm)\Big).
\end{equation}
\end{lemma}

\begin{proof}
Let us start observing that
\[
\begin{split}
T\cons(\mu,\nu) & = T\int_0^1 \cons(\mu,\nu)\De t = \iint_0^1\Big(\frac{1}{2T} |v_t^T|^2 - \frac{T}{8} |\nabla\log\rho_t^T|^2\Big)\rho_t^T\,\De t\De\mm \\
& = \cost(\mu,\nu) -\frac12\Big(\ent(\mu\,|\,\mm) + \ent(\nu\,|\,\mm)\Big) - \frac{T}{4} \iint_0^1 |\nabla\log\rho_t^T|^2\rho_t^T\,\De t\De\mm,
\end{split}
\]
whence trivially the desired upper bound for $T\cons(\mu,\nu)$. On the other hand, Young's inequality and \eqref{eq:cost representations} yield
\[
\begin{split}
\frac{T}{4}\iint_0^1 |\nabla\log\rho_t^T|^2\rho_t^T\,\De t\De\mm & \leq \frac{T}{2}\iint_0^1\Big(|\nabla\log \sfP_{Tt}f^T|^2 + |\nabla\log \sfP_{T(1-t)}g^T|^2\Big)\rho_t^T\,\De t\De\mm \\
& = 2\cost(\mu,\nu) - \ent(\mu\,|\,\mm) - \ent(\nu\,|\,\mm),
\end{split}
\]
so that plugging this inequality into the previous identity gives also the lower bound for $T\cons(\mu,\nu)$.
\end{proof}

We are now in the position to prove Theorem \ref{thm:longtime}.

\begin{proof}
It is easily seen that the unique optimal coupling in 
\be\label{decoupledSP}
\inf_{\pi\in\Pi(\mu,\nu)}\ent(\pi\,|\,\mm\otimes\mm)
\ee
is $\mu\otimes\nu$ and that the optimal value is $\ent(\mu\,|\,\mm) + \ent(\nu\,|\,\mm)$. Consider a sequence $(T_n)_{n\in\N}$ such that $T_n\to\infty$ and let $\pi_n$ be the optimal coupling in the corresponding Schr\"odinger problem, i.e. $\ent(\pi^n\,|\,R_{0,T_n}) = \mathscr{C}_{T_n}(\mu,\nu)$. Since $\Pi(\mu,\nu)$ is compact for the weak topology, $\pi^n$ has an accumulation point $\pi^*$. The basic results of $\Gamma$-convergence ensure that $\pi^*$ is optimal for \eqref{decoupledSP}, whence $\pi^* = \mu\otimes\nu$ by uniqueness. Using again the basic properties of $\Gamma$-convergent sequences we finally obtain the convergence of the optimal values, namely 
\[
\lim_{n\to\infty}\mathscr{C}_{T_n}(\mu,\nu) = \ent(\mu\,|\,\mm) + \ent(\nu\,|\,\mm).
\]
Dividing by $T$ \eqref{eq:cost bounds cons} and letting $T \to \infty$, the long-time behavior of $\cons(\mu,\nu)$ is established as well.
\end{proof}

\subsection{Proof of Theorem \ref{thm:derivcost} and Theorem \ref{thm short time}}

The proof of Theorem \ref{thm:derivcost} requires the preparatory Lemmas \ref{lem:incr fish}, \ref{lem:fg unif conv} and \ref{lem:energy is continuous}. In the first one we show that the Fisher information of the entropic interpolation is non-increasing as a function of $T$.

\begin{lemma}\label{lem:incr fish}
Under the assumptions of Theorem \ref{thm:derivcost}, let $(\rho^T,v^T)$ be optimal for the formulation \eqref{eq:bbs2} and $(\rho^0,v^0)$ for \eqref{eq:benamou-brenier}. Then the function
\bes
[0,\infty)\ni T \mapsto \iint_0^1 |\nabla \log \rho^T_{t}|^2 \rho^T_t \De t\De \mm
\ees
is non-increasing, where $\int |\nabla\log\rho_t^0|^2 \rho_t^0\De\mm := +\infty$ whenever $\rho_t^0 \notin W^{1,2}_{loc}(M)$, the space of locally Sobolev functions on $M$.
\end{lemma}

\begin{proof}
Fix $0 \leq T_1 < T_2$. Summing the inequalities
\bes
\mathscr{A}_{T_1}(\rho^{T_1},v^{T_1}) \leq \mathscr{A}_{T_1}(\rho^{T_2},v^{T_2}), \quad \mathscr{A}_{T_2}(\rho^{T_2},v^{T_2}) \leq \mathscr{A}_{T_2}(\rho^{T_1},v^{T_1})
\ees
and dividing by $(T_2^2-T_1^2)/8$ we obtain 
\bes
\iint_0^1 |\nabla \log \rho^{T_2}_t|^2 \rho^{T_2}_t \De t \De \mm \leq \iint_0^1 |\nabla \log \rho^{T_1}_t|^2 \rho^{T_1}_t \De t \De \mm,
\ees
which is the desired conclusion.
\end{proof}

In the second lemma we prove the continuity in $T$ of the functions $f^T$ and $g^T$ given by \eqref{eq:schr system} with respect to the $L^p(\mm)$ norm, for any $p \in [1,\infty)$.

\begin{lemma}\label{lem:fg unif conv}
Under the assumptions of Theorem \ref{thm:derivcost}, the functions $T \mapsto f^T$ and $T \mapsto g^T$ are continuous from $(0,\infty)$ to $L^p(\mm)$, for any $p \in [1,\infty)$.

As a byproduct, the functions $T \mapsto \sfP_T f^T$ and $T \mapsto \sfP_T g^T$ are continuous from $(0,\infty)$ to $L^p(\mm)$, $p \in [1,\infty)$, as well.
\end{lemma}

\begin{proof}
Let $T_0 > 0$, $0<\delta<T_0$ and denote by $C_\delta$ the positive constant provided by Lemma \ref{lem:fg bounds}-(i) on the interval $[T_0-\delta,T_0+\delta]$. From the first bound in \eqref{eq:fg local bounds} we immediately deduce that $(f^T)_{T \in [T_0-\delta,T_0+\delta]}$ is bounded in $L^\infty(\mm)$, hence in $L^p(\mm)$ for all $p \in [1,\infty]$, because all the functions $f^T$ are supported in $\supp(\mu)$ and this has finite mass, because bounded. To prove the same property for $(g^T)_{T \in [T_0-\delta,T_0+\delta]}$ requires a more technical argument. 

From \eqref{eq:upperbound} with $\varepsilon = 1$ we first oberve that
\[
\begin{split}
1 & = \int f^T(x)g^T(y)\hp_T(x,y)\mm(\De x)\mm(\De y) \\
& \leq \int f^T(x)g^T(y)\frac{C(\kappa,T_0,\delta)}{\sqrt{\mm(B_{\sqrt{T_0-\delta}}(x))\mm(B_{\sqrt{T_0-\delta}}(y))}}\exp\Big(-\frac{\sfd^2(x,y)}{5(T_0+\delta)}\Big) \mm(\De x)\mm(\De y) \\
& \leq \int f^T(x)g^T(y)\frac{C(\kappa,T_0,\delta)}{\sqrt{\mm(B_{\sqrt{T_0-\delta}}(x))\mm(B_{\sqrt{T_0-\delta}}(y))}} \mm(\De x)\mm(\De y), \qquad \forall T \in [T_0-\delta,T_0+\delta]
\end{split}
\]
and since $\supp(f^T) \subset \supp(\mu)$, $\supp(g^T) \subset \supp(\nu)$ for all $T>0$, $\mu,\nu$ have bounded support and $x \mapsto \mm(B_t(x))$ is a continuous function,
\[
\inf_{x \in \supp(f^T)}\mm(B_{\sqrt{T_0-\delta}}(x),\, \inf_{y \in \supp(g^T)}\mm(B_{\sqrt{T_0-\delta}}(y) \geq C_\delta > 0 
\]
with $C_\delta$ independent of $T$. Thus if we combine this inequality with the previous one and recall the chosen normalization \eqref{eq:gauge} we obtain $1 \leq C'_\delta \|f^T\|_{L^1(\mm)}$ for all $T \in [T_0-\delta,T_0+\delta]$ for some $C'_\delta > 0$ (depending on $T_0$ and $\kappa$ too), i.e.\
\[
\|f^T\|_{L^1(\mm)} \geq \frac{1}{C'_\delta}, \qquad \forall T \in [T_0-\delta,T_0+\delta].
\]
Plugging this inequality into the second bound in \eqref{eq:fg local bounds} yields
\begin{equation}\label{eq:muoio di caldo}
\|g^T\|_{L^\infty(\mm)} \leq \frac{C'_\delta}{C_\delta}\|\sigma\|_{L^\infty(\mm)}, \qquad \forall T \in [T_0-\delta,T_0+\delta].
\end{equation}
Hence also $(g^T)_{T \in [T_0-\delta,T_0+\delta]}$ is bounded in $L^p(\mm)$ for all $p \in [1,\infty]$, because all the functions $g^T$ are supported in $\supp(\nu)$ and this has finite mass. 

After this premise, we are now in the position to prove the continuity of $T \mapsto f^T$ and $T \mapsto g^T$. Indeed, if $|h|<\delta$ then by Banach-Alaoglu theorem there exist a (not relabeled) subsequence and functions $\tilde{f},\tilde{g} \in L^1 \cap L^\infty(\mm)$ such that $f^{T_0+h} \stackrel{*}{\rightharpoonup} \tilde{f}$ and $g^{T_0+h} \stackrel{*}{\rightharpoonup} \tilde{g}$ in $L^\infty(\mm)$ as $h \to 0$. We claim that
\begin{equation}\label{eq:claim}
\lim_{h \to 0}\sfP_{T_0+h}f^{T_0+h} = \sfP_{T_0}\tilde{f} \qquad \text{and} \qquad \lim_{h \to 0}\sfP_{T_0+h}g^{T_0+h} = \sfP_{T_0}\tilde{g},
\end{equation}
where both limits have to be understood in $L^p(\mm)$ for $p \in [1,\infty)$.

To this aim, write
\[
\|\sfP_{T_0+h}f^{T_0+h} - \sfP_{T_0}\tilde{f}\|_{L^p(\mm)} \leq \|\sfP_{T_0+h}f^{T_0+h} - \sfP_{T_0+h}\tilde{f}\|_{L^p(\mm)} + \|\sfP_{T_0+h}\tilde{f} - \sfP_{T_0}\tilde{f}\|_{L^p(\mm)}
\]
and observe that the second term on the right-hand side trivially vanishes as $h \to 0$. As regards the first one,
\[
\begin{split}
\|\sfP_{T_0+h}f^{T_0+h} - \sfP_{T_0+h}\tilde{f}\|_{L^p(\mm)} & = \|\sfP_{\delta + h}(\sfP_{T_0-\delta}f^{T_0+h} - \sfP_{T_0-\delta}\tilde{f})\|_{L^p(\mm)} \\
& \leq \|\sfP_{T_0-\delta}f^{T_0+h} - \sfP_{T_0-\delta}\tilde{f}\|_{L^p(\mm)},
\end{split}
\]
because $\sfP_T : L^p(\mm) \to L^p(\mm)$ is a contraction for all $p \in [1,\infty]$. Secondly, by \eqref{eq:heat representation}, the fact that $f^{T_0+h} \stackrel{*}{\rightharpoonup} \tilde{f}$ in $L^\infty(\mm)$ and $\hp_t(x,\cdot) \in L^1(\mm)$ for all $x \in M$, we deduce that 
\begin{equation}
\label{eq:pointwise}
\lim_{h \to 0}\sfP_{T_0-\delta}f^{T_0+h}(x) = \sfP_{T_0-\delta}\tilde{f}(x), \qquad \forall x \in M.
\end{equation}
Since, as already remarked, $(f^T)_{T \in [T_0-\delta,T_0+\delta]}$ is bounded in $L^\infty(\mm)$ and all the functions $f^T$ are supported in $\supp(\mu)$, there exists $M>0$ sufficiently large such that $0 \leq f^T \leq M\mathds{1}_{\supp(\mu)}$ for all $T \in [T_0-\delta,T_0+\delta]$, whence
\[
0 \leq \sfP_{T_0-\delta}f^T \leq M \sfP_{T_0-\delta}(\mathds{1}_{\supp(\mu)}), \qquad \forall T \in [T_0-\delta,T_0+\delta]
\]
by the maximum principle. As the right-hand side above belongs to $L^1 \cap L^\infty(\mm)$ and does not depend on $T$, by \eqref{eq:pointwise} and the dominated convergence theorem we thus infer that
\[
\lim_{h \to 0}\|\sfP_{T_0-\delta}f^{T_0+h} - \sfP_{T_0-\delta}\tilde{f}\|_{L^p(\mm)} = 0,
\]
whence, combining this fact with the previous steps, the validity of the first limit in \eqref{eq:claim}. The argument for the second limit is completely analogous.

As a consequence, up to extract a further (not relabeled) subsequence, we have that the limits in \eqref{eq:claim} hold $\mm$-a.e. Therefore, if we look at the Schr\"odinger system \eqref{eq:schr system} at time $T_0+h$, which reads as
\[
\rho = f^{T_0+h}\sfP_{T_0+h}g^{T_0+h}, \qquad \sigma = g^{T_0+h}\sfP_{T_0+h}f^{T_0+h},
\]
and rewrite it as
\[
f^{T_0+h} = \frac{\rho}{\sfP_{T_0+h}g^{T_0+h}}, \qquad g^{T_0+h} = \frac{\sigma}{\sfP_{T_0+h}f^{T_0+h}},
\]
which is possible because $\sfP_{T_0+h}f^{T_0+h}, \sfP_{T_0+h}g^{T_0+h} > 0$, we deduce that
\[
\lim_{h \to 0}f^{T_0+h} = \frac{\rho}{\sfP_{T_0}\tilde{g}} \qquad \textrm{and} \qquad \lim_{h \to 0}g^{T_0+h} = \frac{\sigma}{\sfP_{T_0}\tilde{f}},
\]
both limits being understood $\mm$-a.e. This is compatible with $f^{T_0+h} \stackrel{*}{\rightharpoonup} \tilde{f}$ and $g^{T_0+h} \stackrel{*}{\rightharpoonup} \tilde{g}$ in $L^\infty(\mm)$ as $h \to 0$ if and only if
\[
\tilde{f} = \frac{\rho}{\sfP_{T_0}\tilde{g}} \qquad \textrm{and} \qquad \tilde{g} = \frac{\sigma}{\sfP_{T_0}\tilde{f}},
\]
namely if and only if
\[
\rho = \tilde{f}\sfP_{T_0}\tilde{g}, \qquad \sigma = \tilde{g}\sfP_{T_0+h}\tilde{f}.
\]
This means that $\tilde{f} = f^{T_0}$ and $\tilde{g} = g^{T_0}$ by uniqueness of the $(f,g)$-decomposition and by the normalization \eqref{eq:gauge}. Since this is true for any convergent subsequence obtained via Banach-Alaoglu theorem, we deduce that the whole sequences $(f^{T_0+h})$ and $(g^{T_0+h})$ converge in $L^p(\mm)$ to $f^{T_0}$ and $g^{T_0}$ respectively, whence the desired conclusion.

It is then sufficient to combine the fact that $\tilde{f} = f^{T_0}$ and $\tilde{g} = g^{T_0}$ with \eqref{eq:claim} to get also the continuity of $T \mapsto \sfP_T f^T$ and $T \mapsto \sfP_T g^T$.
\end{proof}

Finally, in the next lemma we show that the cost $\cost(\mu,\nu)$, the energy $\cons(\mu,\nu)$ and the integrated-in-time Fisher information $\iint_0^1 |\nabla\log\rho_t^T|^2\,\De\mu_t^T$ are continuous as a function of $T$ (the regularity of $\cost(\mu,\nu)$ will be greatly improved in Theorem \ref{thm:derivcost}, of course).

\begin{lemma}
\label{lem:energy is continuous}
Under the assumptions of Theorem \ref{thm:derivcost}, the functions $T \mapsto \cost(\mu,\nu)$, $T \mapsto \cons(\mu,\nu)$ and $T \mapsto \iint_0^1 |\nabla\log\rho_t^T|^2\,\De\mu_t^T$ are continuous on $(0,\infty)$.
\end{lemma}

\begin{proof}
For $T \mapsto \cost(\mu,\nu)$, by \eqref{eq:schr system} we observe that the cost can be rewritten as
\[
\cost(\mu,\nu) = \int (f^T\log f^T)\sfP_T g^T\,\De\mm + \int (g^T\log g^T)\sfP_T f^T\,\De\mm
\]
and by Lemma \ref{lem:fg unif conv}, the locally uniform $L^\infty$-bounds provided by Lemma \ref{lem:fg bounds} and by dominated convergence it is easy to see that the right-hand side above continuously depends on $T$.

As regards $T \mapsto \cons(\mu,\nu)$, let us start observing that the second identity in \eqref{eq:energy representation} can be equivalently rewritten as
\[
\cons(\mu,\nu) = -\frac{1}{2}\int\nabla\sfP_{Tt}f^T \cdot \nabla\sfP_{T(1-t)}g^T\,\De\mm,
\]
so fix $t \in (0,1)$ and note that, by integration by parts first and Cauchy-Schwarz inequality then, we obtain
\[
\begin{split}
|\mathscr{E}_{T+h}(\mu,\nu) - \cons(\mu,\nu)| & \leq \frac{1}{2}\Big|\int \nabla\sfP_{(T+h)t}f^{T+h} \cdot \nabla(\sfP_{(T+h)(1-t)}g^{T+h}-\sfP_{T(1-t)}g^T)\,\De\mm\Big| \\
& \quad + \frac{1}{2}\Big|\int \nabla(\sfP_{(T+h)t}f^{T+h} - \sfP_{Tt}f^T) \cdot \nabla\sfP_{T(1-t)}g^T)\,\De\mm\Big| \\
& \leq \frac{1}{2}\Big|\int \sfL\sfP_{(T+h)t}f^{T+h} (\sfP_{(T+h)(1-t)}g^{T+h}-\sfP_{T(1-t)}g^T)\,\De\mm\Big| \\
& \quad + \frac{1}{2}\Big|\int (\sfP_{(T+h)t}f^{T+h} - \sfP_{Tt}f^T) \sfL\sfP_{T(1-t)}g^T)\,\De\mm\Big| \\
& \leq \frac{1}{2}\|\sfL\sfP_{(T+h)t}f^{T+h}\|_{L^2(\mm)} \|\sfP_{(T+h)(1-t)}g^{T+h}-\sfP_{T(1-t)}g^T\|_{L^2(\mm)} \\
& \quad + \frac{1}{2} \|\sfP_{(T+h)t}f^{T+h} - \sfP_{Tt}f^T\|_{L^2(\mm)} \|\sfL\sfP_{T(1-t)}g^T\|_{L^2(\mm)}.
\end{split}
\]
The second summand after the last inequality clearly vanishes as $h \to 0$, because of Lemma \ref{lem:fg unif conv}. As concerns the first one, argue as in the proof of Lemma \ref{lem:fg unif conv} and notice that for $h$ sufficiently small, say $|h| < \delta$ with $0 \leq \delta \leq T$, all the functions $f^{T+h}$ are uniformly bounded in $L^\infty(\mm)$ and supported in $\supp(\mu)$, hence also uniformly bounded in $L^2(\mm)$. By the a priori estimate \eqref{eq:apriori}, this is sufficient to conclude that there exists $M>0$ sufficiently large such that $\|\sfL\sfP_{(T+h)t}f^{T+h}\|_{L^2(\mm)} \leq M$ for all $|h|<\delta$. Hence, by Lemma \ref{lem:fg unif conv}, also the first summand after the last inequality converges to 0 in the limit, whence the continuity of $T \mapsto \cons(\mu,\nu)$.

Finally, the continuity of $T \mapsto \iint_0^1 |\nabla\log\rho_t^T|^2\,\De\mu_t^T$ is a pure matter of algebraic manipulation, since by \eqref{eq:bbs2} and \eqref{eq:energy representation}
\[
\iint_0^1 |\nabla\log\rho_t^T|^2\,\De\mu_t^T = \frac{4}{T}\cost(\mu,\nu) - 4\cons(\mu,\nu) - \frac{2}{T}\Big(\ent(\mu\,|\,\mm) + \ent(\nu\,|\,\mm)\Big)
\]
and by the previous discussion the right-hand side is continuous on $(0,\infty)$.
\end{proof}

Let us now prove Theorem \ref{thm:derivcost}.

\begin{proof}[Proof of Theorem \ref{thm:derivcost}]
Consider the map $T \mapsto \mathscr{A}_T(\rho^T,v^T)$, let $h>0$, write
\begin{equation}\label{eq:difference quotient}
\begin{split}
\frac{\mathscr{A}_{T+h}(\rho^{T+h},v^{T+h}) - \mathscr{A}_T(\rho^T,v^T)}{h} & = \frac{\mathscr{A}_{T+h}(\rho^{T+h},v^{T+h}) - \mathscr{A}_{T+h}(\rho^T,v^T)}{h} \\ & \quad + \frac{\mathscr{A}_{T+h}(\rho^T,v^T) - \mathscr{A}_T(\rho^T,v^T)}{h}
\end{split}
\end{equation}
and observe that the first term on the right-hand side is non-positive by optimality of $(\rho^{T+h},v^{T+h})$ for $\mathscr{A}_{T+h}$. For the second one
\[
\mathscr{A}_{T+h}(\rho^T,v^T) - \mathscr{A}_T(\rho^T,v^T) = \Big(\frac{Th}{4} + \frac{h^2}{8}\Big)\iint_0^1 |\nabla\log\rho^T_t|^2 \rho^T_t\De t\De\mm,
\]
so that
\begin{equation}\label{eq:partial derivative}
\lim_{h \downarrow 0}\frac{\mathscr{A}_{T+h}(\rho^T,v^T) - \mathscr{A}_T(\rho^T,v^T)}{h} = \frac{T}{4}\iint_0^1 |\nabla\log\rho^T_t|^2 \rho^T_t\De t\De\mm = (\partial_T\mathscr{A}_T)(\rho^T,v^T),
\end{equation}
where $\partial_T\mathscr{A}_T$ denotes the partial derivative w.r.t.\ $T$ of $(T,\rho,v) \mapsto \mathscr{A}_T(\rho,v)$. Combining these two facts we deduce
\begin{equation}\label{eq:limsup}
\limsup_{h \downarrow 0}\frac{\mathscr{A}_{T+h}(\rho^{T+h},v^{T+h}) - \mathscr{A}_T(\rho^T,v^T)}{h} \leq (\partial_T\mathscr{A}_T)(\rho^T,v^T).
\end{equation}
On the other hand we can also write
\begin{equation}\label{eq:difference quotient 2}
\begin{split}
\frac{\mathscr{A}_{T+h}(\rho^{T+h},v^{T+h}) - \mathscr{A}_T(\rho^T,v^T)}{h} & = \frac{\mathscr{A}_{T+h}(\rho^{T+h},v^{T+h}) - \mathscr{A}_T(\rho^{T+h},v^{T+h})}{h} \\ & \quad + \frac{\mathscr{A}_T(\rho^{T+h},v^{T+h}) - \mathscr{A}_T(\rho^T,v^T)}{h}
\end{split}
\end{equation}
and remark that now the second term on the right-hand side is non-negative by optimality of $(\rho^T,v^T)$ for $\mathscr{A}_T$. As regards the first one,
\[
\mathscr{A}_{T+h}(\rho^{T+h},v^{T+h}) - \mathscr{A}_T(\rho^{T+h},v^{T+h}) = \Big(\frac{Th}{4} + \frac{h^2}{8}\Big)\iint_0^1 |\nabla\log\rho^{T+h}_t|^2 \rho^{T+h}_t\De t\De\mm
\]
and by Lemma \ref{lem:energy is continuous} we infer that
\begin{equation}\label{eq:partial derivative 2}
\begin{split}
\lim_{h \downarrow 0}\frac{\mathscr{A}_{T+h}(\rho^{T+h},v^{T+h}) - \mathscr{A}_T(\rho^{T+h},v^{T+h})}{h} & \geq \frac{T}{4}\iint_0^1 |\nabla\log\rho^T_t|^2 \rho^T_t\De t\De\mm \\
& = (\partial_T\mathscr{A}_T)(\rho^T,v^T).
\end{split}
\end{equation}
This yields
\[
\liminf_{h \downarrow 0}\frac{\mathscr{A}_{T+h}(\rho^{T+h},v^{T+h}) - \mathscr{A}_T(\rho^T,v^T)}{h} \geq (\partial_T\mathscr{A}_T)(\rho^T,v^T),
\]
which together with \eqref{eq:limsup} implies that $T \mapsto \mathscr{A}_T(\rho^T,v^T)$ is everywhere right differentiable on $(0,\infty)$. Left differentiability follows by an analogous argument. Indeed if $h<0$, then the first term on the right-hand side in \eqref{eq:difference quotient} is non-negative and \eqref{eq:partial derivative} still holds true with $h \uparrow 0$ instead of $h \downarrow 0$, whence
\[
\liminf_{h \uparrow 0}\frac{\mathscr{A}_{T+h}(\rho^{T+h},v^{T+h}) - \mathscr{A}_T(\rho^T,v^T)}{h} \geq (\partial_T\mathscr{A}_T)(\rho^T,v^T).
\]
Applying the same considerations to \eqref{eq:difference quotient 2} and \eqref{eq:partial derivative 2} gives
\[
\limsup_{h \uparrow 0}\frac{\mathscr{A}_{T+h}(\rho^{T+h},v^{T+h}) - \mathscr{A}_T(\rho^T,v^T)}{h} \leq (\partial_T\mathscr{A}_T)(\rho^T,v^T)
\]
and combining this inequality with the previous one entails the desired left differentiability. Therefore $T \mapsto \mathscr{A}_T(\rho^T,v^T)$ is everywhere differentiable on $(0,\infty)$, a fortiori continuous therein, and
\[
\frac{\De}{\De T} \mathscr{A}_T(\rho^T,v^T) = (\partial_T\mathscr{A}_T)(\rho^T,v^T).
\]
Therefore $T \mapsto T\cost(\mu,\nu)$ is everywhere differentiable (hence continuous) on $(0,\infty)$ as well and, by \eqref{eq:bbs2} and the very definition of $\cost(\mu,\nu)$ and $\cons(\mu,\nu)$,
\[
\begin{split}
\frac{\De}{\De T}\big(T\cost(\mu,\nu)\big) & = \frac12\Big(\ent(\mu\,|\,\mm) + \ent(\nu\,|\,\mm)\Big) + (\partial_T\mathscr{A}_T)(\rho^T,v^T) \\
& = \frac12\Big(\ent(\mu\,|\,\mm) + \ent(\nu\,|\,\mm)\Big) + \frac{T}{4}\iint_{0}^1|\nabla \log \rho^T_t|^2 \rho^{T}_t \De \mm \\
& = \cost(\mu,\nu) - T\cons(\mu,\nu),
\end{split}
\]
which proves both formulas for the first derivative of $T \mapsto T\cost(\mu,\nu)$, thanks to \eqref{eq:reparametrization}. Now notice that the right-hand side above is continuous on $(0,\infty)$: $T \mapsto \cost(\mu,\nu)$ is continuous on $(0,\infty)$ since so is $T \mapsto T\cost(\mu,\nu)$, while $T \mapsto T\cons(\mu,\nu)$ is continuous by Lemma \ref{lem:energy is continuous}. Hence $T \mapsto T\cost(\mu,\nu)$ belongs to $C^1((0,\infty))$. The fact that $T \mapsto T\cost(\mu,\nu)$ is also twice differentiable a.e.\ follows from the fact that
\begin{equation}\label{eq:un'altra formula}
\frac{\De}{\De T}\big(T\cost(\mu,\nu)\big) = \frac12\Big(\ent(\mu\,|\,\mm) + \ent(\nu\,|\,\mm)\Big) + \frac{T}{4}\iint_{0}^1|\nabla \log \rho^T_t|^2 \rho^{T}_t \De \mm
\end{equation}
and the last term on the right-hand side is the product of a linear function and a monotone (non-increasing) one, thanks to Lemma \ref{lem:incr fish}.

This concludes the proof of (ii). Claim (i) is a straightforward consequence as well as the formula for the second derivative of $T \mapsto T\cost(\mu,\nu)$.
\end{proof}

From Theorem \ref{thm:derivcost}, the proof of Theorem \ref{thm short time} follows.

\begin{proof}[Proof of Theorem \ref{thm short time}]
As regards \eqref{short time bound}, we only prove the upper bound, the lower bound being an immediate consequence of the Benamou-Brenier formulation \eqref{eq:bbs2}. Since $T\mapsto T \cost(\mu,\nu)$ belongs to $C^1((0,\infty))$ we can use the fundamental theorem of calculus that gives for all $\varepsilon>0$
\[
\begin{split}
T\cost(\mu,\nu)-\ve \mathscr{C}_{\ve}(\mu,\nu) & = \frac{\cH(\mu\,|\,\mm) +\cH(\nu\,|\,\mm)}{2}(T-\ve) + \int_{\ve}^{T} \frac{S}{4}\Big( \iint_0^1|\nabla \log \rho^S_t|^2\rho^S_t\De t\De\mm \Big) \De S \\
& \leq \frac{\cH(\mu\,|\,\mm) + \cH(\nu\,|\,\mm)}{2}(T-\ve) + \Big(\frac{T^2}{8} - \frac{\ve^2}{8}\Big) \iint_0^1|\nabla \log \rho^0_t|^2\rho^0_t\De t\De\mm,
\end{split}
\]
where the inequality is motivated by Lemma \ref{lem:incr fish}. The conclusion follows letting $\varepsilon\rightarrow 0$ and using the convergence of the rescaled entropic cost towards $W^2_2(\mu,\nu)/2$. 

In order to prove \eqref{cost taylor}, assume $\iint_0^1|\nabla \log \rho^0_t|^2\rho^0_t\De t\De\mm < \infty$ and accept, for the moment, that the following three auxiliary facts hold: (a) $T \mapsto \iint_0^1|\nabla \log \rho^T_t|^2\rho^T_t\De t\De\mm$ is continuous at $T=0$; (b) the first derivative of $T \mapsto T\cost(\mu,\nu)$ is continuous up to $T=0$; (c) $T \mapsto T\cost(\mu,\nu)$ is twice differentiable at $T=0$. Then (c) implies that $T \mapsto T\cost(\mu,\nu)$ has a Taylor expansion of the form
\[
T\cost(\mu,\nu) = \frac{1}{2}W_2^2(\mu,\nu) + A T + B \frac{T^2}{2}+o(T^2).
\]
The coefficients can be identified thanks to Theorem \ref{thm:derivcost}. Indeed
\[
A = \frac{\De}{\De T}\Big|_{T=0^+}(T\cost(\mu,\nu)) = \lim_{T' \downarrow 0}\frac{\De}{\De T}\Big|_{T=T'}(T\cost(\mu,\nu)) = \frac{1}{2}\Big(\ent(\mu\,|\,\mm) + \ent(\nu\,|\,\mm)\Big),
\]
the second identity being a consequence of (b), the third one following from \eqref{eq:un'altra formula} and the finiteness of $\iint_0^1|\nabla \log \rho^0_t|^2\rho^0_t\De t\De\mm$. As regards $B$, this is given by
\[
\begin{split}
B & = \frac{\De^2}{\De T^2} \Big|_{T=0^+}(T\cost(\mu,\nu)) = \lim_{T' \downarrow 0}\frac{1}{T'}\Big(\frac{\De}{\De T} \Big|_{T=T'}(T\cost(\mu,\nu)) - \frac{\De}{\De T} \Big|_{T=0^+}(T\cost(\mu,\nu))\Big) \\
& = \lim_{T' \downarrow 0} \frac{1}{4}\iint_0^1|\nabla \log \rho^{T'}_t|^2\rho^{T'}_t\De t\De\mm = \frac{1}{4}\iint_0^1|\nabla \log \rho^0_t|^2\rho^0_t\De t\De\mm
\end{split}
\]
thanks to (a). Therefore it only remains to show (a), (b), and (c); we shall start with the first one. On the one hand, by Lemma \ref{lem:incr fish}
\begin{equation}\label{eq:limsup fisher}
\limsup_{T \downarrow 0} \iint_0^1|\nabla \log \rho^T_t|^2\rho^T_t\De t\De\mm \leq \iint_0^1|\nabla \log \rho^0_t|^2\rho^0_t\De t\De\mm.
\end{equation}
On the other hand, as $\CD(\kappa,N)$ holds, we can and shall rely on powerful facts established in \cite{GigTam18}: by Propositions 5.1 and 5.4 therein we know that 
\begin{equation}\label{eq:weak star}
\rho_t^T \stackrel{*}{\rightharpoonup} \rho_t^0 \quad \textrm{in } L^\infty(\mm), \qquad \forall t \in [0,1]
\end{equation}
and by \cite[Lemma 4.2]{GigTam18} we also know that for all $\bar{x} \in M$ there exist $C',C'',r > 0$ such that
\begin{equation}\label{eq:exponential decrease}
\rho_t^T(x) \leq C'\exp\Big(-\frac{C''\sfd^2_g(x,\bar{x})}{T}\Big), \qquad \forall x \in M \setminus B_r(\bar{x}),\, t \in [0,1],\, T \in (0,1).
\end{equation}
With this said, let $B \subset M$ be any Borel set and let us prove that
\begin{equation}\label{eq:intermediate}
(\rho_t^T\mm)(B) \to (\rho_t^0\mm)(B) \quad \textrm{as } T \downarrow 0, \qquad \forall t \in [0,1].
\end{equation}
To this end, since $\mu,\nu$ are compactly supported, there exists a bounded set containing the supports of $\rho_t^0$ for all $t \in [0,1]$, so that up to choose a larger $r$ we can assume that $\supp(\rho_t^0) \subset B_r(\bar{x})$ for all $t \in [0,1]$. By \eqref{eq:weak star} and the trivial fact that $\mathds{1}_{B \cap B_r(\bar{x})} \in L^1(\mm)$ this implies that $(\rho_t^T\mm)(B \cap B_r(\bar{x})) \to (\rho_t^0\mm)(B \cap B_r(\bar{x}))$ as $T \downarrow 0$ for all $t \in [0,1]$. By \eqref{eq:exponential decrease} and the fact that under the $\CD(\kappa,N)$ condition the function in the right-hand side of \eqref{eq:exponential decrease} is integrable and converges monotonically to 0 as $T \downarrow 0$, it also follows 
\[
(\rho_t^T\mm)(B \setminus B_r(\bar{x})) \to 0 = (\rho_t^0\mm)(B \setminus B_r(\bar{x})) \quad \textrm{as } T \downarrow 0,
\]
whence \eqref{eq:intermediate}. By \cite[Theorem 7.6]{AmbrosioGigliSavare11} this implies
\[
\liminf_{T \downarrow 0} \int |\nabla \log \rho^T_t|^2\rho^T_t \De\mm \geq \int |\nabla \log \rho^0_t|^2\rho^0_t\De t\De\mm, \qquad \forall t \in [0,1]
\]
whence by Fatou's lemma
\[
\liminf_{T \downarrow 0}\iint_0^1|\nabla \log \rho^T_t|^2\rho^T_t\De t\De\mm \geq \iint_0^1|\nabla \log \rho^0_t|^2\rho^0_t\De t\De\mm.
\]
Combining this inequality with \eqref{eq:limsup fisher} yields the desired continuity at $T=0$ and thus (a).

As regards (b), by \eqref{short time bound} and the finiteness of $\iint_0^1|\nabla \log \rho^0_t|^2\rho^0_t\De t\De\mm$ it immediately follows that
\begin{equation}\label{eq:quasi finito}
\limsup_{T \downarrow 0} \Big\{\cost(\mu,\nu) - \frac{1}{2T}W_2^2(\mu,\nu)\Big\} \leq \frac{1}{2}\Big(\ent(\mu\,|\,\mm) + \ent(\nu\,|\,\mm)\Big).
\end{equation}
On the other hand, by following the argument of Theorem \ref{thm:derivcost} and recalling \eqref{eq:benamou-brenier} we see that
\[
\begin{split}
T\cost(\mu,\nu) - \frac{1}{2}W_2^2(\mu,\nu) & = \frac{T}{2}\Big(\ent(\mu\,|\,\mm) + \ent(\nu\,|\,\mm)\Big) + \mathscr{A}_T(\rho^T,v^T) - \mathscr{A}_0(\rho^0,v^0) \\
& \geq \frac{T}{2}\Big(\ent(\mu\,|\,\mm) + \ent(\nu\,|\,\mm)\Big) + \mathscr{A}_T(\rho^T,v^T) - \mathscr{A}_0(\rho^T,v^T) \\
& = \frac{T}{2}\Big(\ent(\mu\,|\,\mm) + \ent(\nu\,|\,\mm)\Big) + \frac{T^2}{8}\iint_0^1 |\nabla\log\rho_t^T|^2\rho_t^T \De t\De\mm,
\end{split}
\]
where $(v_t^0)$ is any choice of gradient of intermediate Kantorovich potentials associated with the geodesic $(\rho_t^0\mm)$, and by (a) this implies
\[
\liminf_{T \downarrow 0} \Big\{\cost(\mu,\nu) - \frac{1}{2T}W_2^2(\mu,\nu)\Big\} \geq \frac{1}{2}\Big(\ent(\mu\,|\,\mm) + \ent(\nu\,|\,\mm)\Big).
\]
Combining this inequality with \eqref{eq:quasi finito} yields the differentiability of $T \mapsto T\cost(\mu,\nu)$ at $T=0$ and by \eqref{eq:un'altra formula} we see that the derivative is continuous there. Finally, by \eqref{eq:bbs2}, the computations above for determining the coefficients $A$ and $B$, and the properties (a) and (b) it is not difficult to see that (c) holds too.
\end{proof}

\subsection{Proof of Theorem \ref{thm:longtime} and Corollary \ref{cor:logsob}}\label{subsec:strong}

In this section we provide the reader with a second proof of Theorem \ref{thm:longtime}, whose crucial ingredient is a long-time analogue of Lemma \ref{lem:fg unif conv}. This is precisely the content of the next result. With respect to the proof given in Section \ref{subsec:gamma}, here the marginals $\mu,\nu$ satisfy \eqref{hyp marginals weak}, hence a condition stronger than \eqref{hyp marginals weakest}; as a consequence we gain information on the long-time behavior of $f^T$ and $g^T$ separately, determining their limits as $T \to \infty$. Let us point out again that \eqref{hyp ref} below is required only when $\kappa=0$.

\begin{lemma}\label{lem:fg conv at infinity}
Under the assumptions \eqref{hyp pot} with $\kappa \geq 0$, \eqref{hyp ref} and \eqref{hyp marginals weak}, for the functions $f^T,g^T$ given by \eqref{eq:schr system} it holds
\begin{equation}\label{eq:limit infinity}
\lim_{T \to \infty}f^T = \rho, \qquad \lim_{T \to \infty}g^T = \sigma
\end{equation}
where both limits are in $L^p(\mm)$ for any $p \in [1,\infty)$.

As a byproduct
\begin{equation}\label{eq:limit infinity 2}
\lim_{T \to \infty}\sfP_T f^T = \lim_{T \to \infty}\sfP_T g^T = 1,
\end{equation}
where both limits are in $L^p(\mm)$ for any $p \in [1,\infty)$.
\end{lemma}

\begin{proof}
Let $T_0 > 0$ and denote by $C$ the positive constant provided by Lemma \ref{lem:fg bounds}-(ii) on the interval $[T_0,\infty)$. From the first bound in \eqref{eq:fg infinity bounds} we immediately deduce that $(f^T)_{T \geq T_0}$ is bounded in $L^\infty(\mm)$, hence in $L^p(\mm)$ for all $p \in [1,\infty]$, as already argued in the proof of Lemma \ref{lem:fg unif conv}. From \eqref{eq:upperbound} with $C_\kappa=0$ (since $\kappa \geq 0$), \eqref{eq:gauge}, the fact that $\mu,\nu$ have bounded support and arguing as in Lemma \ref{lem:fg unif conv} we also have
\[
1 = \int f^T(x)g^T(y)\hp_T(x,y)\De\mm(x)\De\mm(y) \leq C' \|f^T\|_{L^1(\mm)}, \qquad \forall T \geq T_0
\]
for some $C' > 0$, whence 
\begin{equation}
\label{eq:jaures}
\|f^T\|_{L^1(\mm)} \geq \frac{1}{C}, \qquad \forall T \geq T_0. 
\end{equation}
Plugging this inequality into the second bound in \eqref{eq:fg infinity bounds} yields that also $(g^T)_{T \geq T_0}$ is bounded in $L^\infty(\mm)$, thus in $L^p(\mm)$ for all $p \in [1,\infty]$. As a direct consequence of this and of the maximum, we deduce that also $(\sfP_T f^T)_{T \geq T_0}$ and $(\sfP_T g^T)_{T \geq T_0}$ are bounded in $L^\infty(\mm)$.

We are now in the position to prove \eqref{eq:limit infinity}. Indeed, by Banach-Alaoglu theorem there exist a (not relabeled) subsequence and functions $\tilde{f},\tilde{g} \in L^\infty(\mm)$ such that $f^T \stackrel{*}{\rightharpoonup} \tilde{f}$ and $g^T \stackrel{*}{\rightharpoonup} \tilde{g}$ in $L^\infty(\mm)$ as $T \to \infty$. We claim that
\begin{equation}\label{eq:claim2}
\lim_{T \to \infty}\sfP_T f^T = \int \tilde{f}\,\De\mm \qquad \text{and} \qquad \lim_{T \to \infty}\sfP_T g^T = \int \tilde{g}\,\De\mm,
\end{equation}
where both limits have to be understood in $L^p(\mm)$ for $p \in [1,\infty)$. To this aim observe that
\begin{equation}\label{eq:king gizzard}
\begin{split}
\|\sfP_T f^T - \int \tilde{f}\,\De\mm\|_{L^p(\mm)} & \leq \|\sfP_T f^T - \sfP_T\tilde{f}\|_{L^p(\mm)} + \|\sfP_T\tilde{f} - \int \tilde{f}\,\De\mm\|_{L^p(\mm)} \\
& \leq \|\sfP_{T_0}f^T - \sfP_{T_0}\tilde{f}\|_{L^p(\mm)} + \|\sfP_T\tilde{f} - \int \tilde{f}\,\De\mm\|_{L^p(\mm)},
\end{split}
\end{equation}
where the second inequality is motivated by the fact that $\sfP_T = \sfP_{T-T_0}\circ \sfP_{T_0}$ and $\sfP_{T-T_0}$ is a contraction in $L^p(\mm)$. Therefore, arguing as in the proof of Lemma \ref{lem:fg unif conv} we can see that the first term on the right-hand side above vanishes as $T \to \infty$. On the other hand, the ergodicity of $\sfP_T$ and \eqref{hyp ref} entail that
\begin{equation}
\label{eq:and the lizard wizard}
\lim_{T \to \infty}\sfP_T\tilde{f} = \int\tilde{f}\,\De\mm, \qquad \text{in } L^2(\mm),
\end{equation}
hence in $L^p(\mm)$ for all $p \in [1,2]$, while for $p>2$ it is sufficient to observe that
\[
\begin{split}
\|\sfP_T\tilde{f} - \int \tilde{f}\,\De\mm\|_{L^p(\mm)}^p & \leq \|\sfP_T\tilde{f} - \int \tilde{f}\,\De\mm\|_{L^\infty(\mm)}^{p-2}\|\sfP_T\tilde{f} - \int \tilde{f}\,\De\mm\|_{L^2(\mm)}^2 \\
& \leq \big(\|\tilde{f}\|_{L^\infty(\mm)} + \|\tilde{f}\|_{L^1(\mm)}\big)^{p-2} \|\sfP_T\tilde{f} - \int \tilde{f}\,\De\mm\|_{L^2(\mm)}^2
\end{split}
\]
and use \eqref{eq:and the lizard wizard}. Therefore, also the second term on the second line on the right-hand side in \eqref{eq:king gizzard} converges to 0 as $T \to \infty$ and this proves the first identity in \eqref{eq:claim2}. An analogous argument leads to the second one. 

As a consequence, up to extract a further (not relabeled) subsequence, the limits in \eqref{eq:claim2} hold $\mm$-a.e.\ and if we pass to the limit as $T \to \infty$ in the Schr\"odinger system \eqref{eq:schr system} we get
\[
\rho = \tilde{f}\int\tilde{g}\,\De\mm, \qquad \sigma = \tilde{g}\int\tilde{f}\,\De\mm
\]
and because of \eqref{eq:gauge} it must hold $\|\tilde{g}\|_{L^1(\mm)} = 1$, whence $\tilde{f} = \rho$ and this in turn implies $\tilde{g} = \sigma$.

It is now sufficient to combine this information with \eqref{eq:claim2} to get \eqref{eq:limit infinity 2}.
\end{proof}

We are now in the position to determine the long-time behavior of the entropic cost.

\begin{proof}[Proof of Theorem \ref{thm:longtime}]
Let us first observe that for all $T_0>0$ there exists $C>0$ such that for all $T \geq T_0$ it holds
\begin{subequations}
\begin{align}
\label{eq:two-sided bound f}
& C^{-1} \leq \sfP_T f^T \leq C, \qquad \mm\textrm{-a.e. in } \supp(\nu), \\
\label{eq:two-sided bound g}
& C^{-1} \leq \sfP_T g^T \leq C, \qquad \mm\textrm{-a.e. in } \supp(\mu).
\end{align}
\end{subequations}
Indeed, both upper bounds have been proven in Lemma \ref{lem:fg conv at infinity}, while the lower ones can be deduced relying on \eqref{eq:lowerbound} and the boundedness of $\supp(\mu)$, $\supp(\nu)$. More precisely, for \eqref{eq:two-sided bound f} there exists $C>0$ such that
\[
\sfP_T g^T(x) = \int g^T(y)\hp_T(x,y)\,\De\mm \geq C \|g^T\|_{L^1(\mm)} = C_\delta, \qquad \forall x \in \supp(\mu),\,\forall T \geq T_0,
\]
the last identity being motivated by \eqref{eq:gauge}. For the same reason and taking \eqref{eq:jaures} into account
\[
\sfP_T f^T(x) = \int f^T(y)\hp_T(x,y)\,\De\mm \geq C \|f^T\|_{L^1(\mm)} \geq \frac{C}{C'}, \qquad \forall x \in \supp(\nu),\,\forall T \geq T_0,
\]
which proves also \eqref{eq:two-sided bound g}. With this said, by the very definition of the entropic cost and its equivalence with \eqref{eq:schr system} we have
\[
\cost(\mu,\nu) = \int\log f^T\,\De\mu + \int\log g^T\,\De\nu = \ent(\mu\,|\,\mm) + \ent(\nu\,|\,\mm) - \int\log \sfP_T f^T\,\De\nu - \int\log \sfP_T g^T\,\De\mu
\]
and by \eqref{eq:two-sided bound f}, \eqref{eq:two-sided bound g}, \eqref{eq:limit infinity 2} and the dominated convergence theorem we can pass to the limit as $T \to \infty$ in the identity above, thus getting
\[
\lim_{T \to \infty}\cost(\mu,\nu) = \ent(\mu\,|\,\mm) + \ent(\nu\,|\,\mm).
\]
\end{proof}

As concerns Corollary \ref{cor:logsob}, let us recall an approximation result, whose proof can be found in \cite[Lemma 3.1]{GLRT19}.

\begin{lemma}\label{lem:approximation}
Let $\mu \in \cP(M)$ with $\ent(\mu\,|\,\mm) < \infty$ and
\[
\int |\nabla\log\Big(\frac{\De\mu}{\De\mm}\Big)|^2\De\mu < \infty.
\]
Then there exists $(\mu_n) \subset \cP(M)$ with $\mu_n = \rho_n\mm$, $\rho_n \in C^\infty_c(M)$ such that $\ent(\mu_n\,|\,\mm) \to \ent(\mu\,|\,\mm)$ and
\[
\int |\nabla\log\Big(\frac{\De\mu_n}{\De\mm}\Big)|^2\De\mu_n \to \int |\nabla\log\Big(\frac{\De\mu}{\De\mm}\Big)|^2\De\mu
\]
as $n \to \infty$.
\end{lemma}

We can now provide an ``entropic'' proof of the logarithmic Sobolev inequality.

\begin{proof}[Proof of Corollary \ref{cor:logsob}]
First of all, we can assume that the right-hand side in \eqref{eq:logsob} is finite, otherwise the statement is trivially true; by Lemma \ref{lem:approximation} we can also assume that $\mu$ has compact support, is absolutely continuous w.r.t.\ $\mm$ and its density $\rho$ is smooth. Hence, if we take any probability measure $\nu = \sigma\mm$ with compact support and smooth density, \eqref{hyp marginals weak} is satisfied. In addition, $f^T$ and $g^T$ are also compactly supported and smooth, as they inherit the regularity of $\rho$ and $\sigma$ respectively.

After this premise, let us show that $t \mapsto \ent(\mu_t^T\,|\,\mm)$ is differentiable at $t=0$ and
\begin{equation}\label{eq:derivative in 0}
\frac{\De}{\De t}\ent(\mu_t^T\,|\,\mm)\bigg|_{t=0} = \int\nabla\log\rho \cdot v_0^T\,\De\mu.
\end{equation}
To this aim, let us recall that from \cite[Lemma 3.7]{GLRT19} for every $\delta>0$ the map $t \mapsto \ent(\mu_t^{T,\delta}\,|\,\mm)$ belongs to $C([0,1]) \cap C^2((0,1))$ and it holds
\begin{equation}
\label{eq:derivative delta}
\frac{\De}{\De t}\ent(\mu_t^{T,\delta}\,|\,\mm) = \int\nabla\log\rho_t^{T,\delta} \cdot v_t^{T,\delta}\,\De\mu_t^{T,\delta},
\end{equation}
where $\rho_t^{T,\delta}$, $\mu_t^{T,\delta}$ and $v_t^{T,\delta}$ are defined as in \eqref{eq:delta definitions}. If we integrate \eqref{eq:derivative delta} on $[0,t]$ with $t \leq 1$ we obtain
\begin{equation}\label{eq:derivative integrated version}
\ent(\mu_t^{T,\delta}\,|\,\mm) - \ent(\mu_0^{T,\delta}\,|\,\mm) = \int_0^t\int \nabla\log\rho_s^{T,\delta} \cdot v_s^{T,\delta}\,\De\mu_s^{T,\delta}\De s
\end{equation}
and by the dominated convergence theorem it is easy to see that the left-hand side converges to $\ent(\mu_t^T\,|\,\mm) - \ent(\mu\,|\,\mm)$ as $\delta \downarrow 0$. As regards the right-hand one, to prove that
\[
\lim_{\delta \downarrow 0} \int_0^t\int \nabla\log\rho_s^{T,\delta} \cdot v_s^{T,\delta}\,\De\mu_s^{T,\delta}\De s = \int_0^t\int \nabla\log\rho_s^T \cdot v_s^T\,\De\mu_s^T\De s
\]
we borrow an argument used in \cite[Lemma 3.11]{GLRT19} and we report it here for reader's sake. By the very definition of $v_t^{T,\delta}$ and since
\[
T\log\rho_t^{T,\delta} = T\log(\sfP_{Tt}f^T + \delta) + T\log(\sfP_{T(1-t)}g^T + \delta),
\]
the desired conclusion is achieved if we are able to prove that
\begin{subequations}
\begin{align}
\label{eq:varphi}
& \lim_{\delta \downarrow 0} \iint_0^t |\nabla\log(\sfP_{Ts}f^T + \delta)|^2\, \rho_s^{T,\delta}\De s\De\mm = \iint_0^t |\nabla\log \sfP_{Ts}f^T|^2\,\rho_s^T\De s\De\mm, \\
\label{eq:psi}
& \lim_{\delta \downarrow 0} \iint_0^t |\nabla\log(\sfP_{T(1-s)}g^T + \delta)|^2\, \rho_s^{T,\delta}\De s\De\mm = \iint_0^t |\nabla\log \sfP_{T(1-s)}g^T|^2\,\rho_s^T\De s\De\mm.
\end{align}
\end{subequations}
To this aim, notice that $\rho_t^{T,\delta} = \rho_t^T + \delta \sfP_{Tt}f^T + \delta \sfP_{T(1-t)}g^T + \delta^2$, whence using either $\sfP_{Tt}f^T + \delta \geq \sfP_{Tt}f^T$ or $\sfP_{Tt}f^T + \delta \geq \delta$ it is easy to infer that
\[
\begin{split}
|\nabla\log(\sfP_{Tt}f^T + \delta)|^2\rho_t^T & = T^2\frac{|\nabla \sfP_{Tt}f^T|^2}{(\sfP_{Tt}f^T + \delta)^2}\rho_t^T \leq T^2\frac{|\nabla \sfP_{Tt}f^T|^2}{(\sfP_{Tt}f^T)^2}\rho_t^T = |\nabla\log \sfP_{Tt}f^T|^2\rho_t^T, \\
\delta|\nabla\log(\sfP_{Tt}f^T + \delta)|^2 \sfP_{Tt}f^T & = \delta T^2\frac{|\nabla \sfP_{Tt}f^T|^2}{(\sfP_{Tt}f^T + \delta)^2} \sfP_{Tt}f^T \leq T^2 |\nabla \sfP_{Tt}f^T|^2, \\
\delta|\nabla\log(\sfP_{Tt}f^T + \delta)|^2 \sfP_{T(1-t)}g^T & = \delta T^2\frac{|\nabla \sfP_{Tt}f^T|^2}{(\sfP_{Tt}f^T + \delta)^2} \sfP_{T(1-t)}g^T \leq T^2\frac{|\nabla \sfP_{Tt}f^T|^2}{\sfP_{Tt}f^T} \sfP_{T(1-t)}g^T \\ & = |\nabla\log \sfP_{Tt}f^T|^2\rho_t^T, \\
\delta^2|\nabla\log \sfP_{Tt}f^T|^2 & = \delta^2 T^2\frac{|\nabla \sfP_{Tt}f^T|^2}{(\sfP_{Tt}f^T + \delta)^2} \leq T^2|\nabla \sfP_{Tt}f^T|^2.
\end{split}
\]
All the right-hand sides above are integrable on $[0,1] \times M$ (either by \eqref{eq:finiteness} or by the Bakry-\'Emery contraction estimate \eqref{eq:bakry-emery} together with $f^T \in C^\infty_c(M)$), thus by dominated convergence \eqref{eq:varphi} follows. An analogous argument holds for \eqref{eq:psi}. Therefore we can pass to the limit as $\delta \downarrow 0$ in \eqref{eq:derivative integrated version} and get
\[
\ent(\mu_t^T\,|\,\mm) - \ent(\mu\,|\,\mm) = \int_0^t\int \nabla\log\rho_s^T \cdot v_s^T\,\De\mu_s^T\De s,
\]
whence \eqref{eq:derivative in 0}. This allows us to differentiate \eqref{eq:k-convexity} at $t=0$ and obtain
\begin{equation}\label{eq:differentiation}
e^{-\kappa T}\ent(\mu\,|\,\mm) \leq -\frac{1-e^{-\kappa T}}{\kappa T}\int \nabla\log\rho \cdot v_0^T\,\De\mu + \ent(\nu\,|\,\mm) - (1-e^{-\kappa T})\cost(\mu,\nu).
\end{equation}
We claim that
\begin{equation}\label{eq:claim3}
    \lim_{T \to \infty}\frac{1}{T}\int \nabla\log\rho \cdot v_0^T\,\De\mu = -\frac{1}{2}\int |\nabla\log\rho|^2 \,\De\mu
\end{equation}
Indeed, 
\[
v_0^T = \frac{T}{2}\nabla\log \sfP_T g^T - \frac{T}{2}\nabla\log f^T = T\nabla\log \sfP_T g^T - \frac{T}{2}\nabla\log\rho
\]
from the very definition of $v_0^T$, so that
\[
\begin{split}
\frac{1}{T}\int \nabla\log\rho\cdot v_0^T\,\De\mu & = \int\nabla\rho\cdot\nabla\log \sfP_T g^T\,\De\mm - \frac{1}{2}\int|\nabla\log\rho|^2\,\De\mm \\
& = -\int (\sfL\rho) \log \sfP_T g^T\,\De\mm - \frac{1}{2}\int|\nabla\log\rho|^2\,\De\mm,
\end{split}
\]
where $\sfL = \Delta/2 - \nabla U \cdot \nabla$ is the generator of $\sfP_t$, self-adjoint w.r.t.\ $\mm$. By \eqref{eq:two-sided bound g}, \eqref{eq:limit infinity 2} and the boundedness of $\supp(\mu)$ we infer that the first integral on the right-hand side vanishes as $T \to \infty$ by dominated convergence, whence \eqref{eq:claim3}.

From \eqref{eq:differentiation}, \eqref{eq:claim3} and Theorem \ref{thm:longtime} the logarithmic Sobolev inequality \eqref{eq:logsob} follows.
\end{proof}

\subsection{Proof of Theorem \ref{thm cost bound}}\label{subsec:cost bound}

The proof of Theorem \ref{thm cost bound} relies on Theorem \ref{thm:derivcost} and two other ingredients: the entropic Talagrand inequality put forward in \cite{Conforti17} and an ``energy-transport'' inequality relating $\cons(\mu,\nu)$ and $\cost(\mu,\nu)$. The former states that if \eqref{hyp pot} holds with $\kappa>0$, then for all $\mu,\nu$ as in \eqref{hyp marginals weakest} and for all $t \in (0,1)$
\begin{equation}
\label{eq:ent-talagrand}
\cost(\mu,\nu) \leq \frac{1}{1-\exp(-\kappa Tt)} \cH(\mu\,|\,\mm)+\frac{1}{1-\exp(-\kappa T(1-t)) }\cH(\nu\,|\,\mm).
\end{equation}
Truth to be told, as for \eqref{eq:k-convexity} also the entropic Talagrand inequality \eqref{eq:ent-talagrand} was stated in \cite{Conforti17} in the framework of compact Riemannian manifolds satisfying \eqref{hyp pot} with $\kappa>0$ and endowed with the volume measure, but since \eqref{eq:ent-talagrand} is deduced from \eqref{eq:k-convexity} and the latter has been generalized to the present framework in Section \ref{sec:preliminaries}, there is no problem in applying it.

The ``energy-transport'' inequality relating $\cons(\mu,\nu)$ and $\cost(\mu,\nu)$ is expressed in the following

\begin{lemma}[Energy-Transport inequality]\label{lem:cons cost}
Assume that \eqref{hyp pot} and \eqref{hyp marginals weak} hold. Then for all $T>0$ it holds
\begin{equation}
\label{eq:sharp}
|\cons(\mu,\nu)| \leq \frac{\kappa}{\exp(\kappa T/2)-1} \sqrt{(\cost(\mu,\nu)-\cH(\mu\,|\,\mm))(\cost(\mu,\nu)-\cH(\nu\,|\,\mm))}
\end{equation}
if $\kappa \neq 0$ and
\begin{equation}
\label{eq:rate}
|\cons(\mu,\nu)| \leq \frac{1}{T} \sqrt{(\cost(\mu,\nu)-\cH(\mu\,|\,\mm))(\cost(\mu,\nu)-\cH(\nu\,|\,\mm))}
\end{equation}
if $\kappa = 0$.
\end{lemma}

\begin{proof}
From the second identity in \eqref{eq:energy representation} we have
\begin{equation}
\label{eq:rough}
|\cons(\mu,\nu)| = \int_0^1|\cons(\mu,\nu)|\De t \leq \frac{1}{2}\iint_0^1 |\nabla\log \sfP_{T(1-t)}g^T| |\nabla\log \sfP_{Tt} f^T|\rho_t^T\,\De t\De\mm,
\end{equation}
so that by Cauchy-Schwarz inequality and by \eqref{eq:cost representations}, \eqref{eq:rate} follows.

As regards \eqref{eq:sharp}, using the same notations as in \eqref{eq:delta definitions}, let us first point out that, being $\rho_t^{T,\delta}$ and $v_t^{T,\delta}$ optimal for SP with marginals $\mu_0^{T,\delta}$ and $\mu_1^{T,\delta}$, we have
\begin{equation}
\label{eq:energy delta}
\cons(\mu_0^{T,\delta},\mu_1^{T,\delta}) = -\frac{1}{2}\int \nabla\log(\sfP_{T(1-t)}g^T+\delta) \cdot \nabla\log(\sfP_{Tt}f^T+\delta)\,\De\mu_t^{T,\delta},
\end{equation}
so that if we define
\[
\Phi(t) := \int |\nabla\log(\sfP_{Tt}f^T + \delta)|^2\,\De\mu_t^{T,\delta}, \qquad \Psi(t) := \int |\nabla\log(\sfP_{T(1-t)}g^T + \delta)|^2\,\De\mu_t^{T,\delta}
\]
for $\delta>0$, then by \eqref{eq:energy delta} evaluated in $t=1/2$ and by Cauchy-Schwarz inequality
\begin{equation}\label{eq:phi psi}
|\cons(\mu_0^{T,\delta},\mu_1^{T,\delta})| \leq \frac{1}{2} \sqrt{\Phi(1/2)\Psi(1/2)}.
\end{equation}
In order to estimate the right-hand side above, observe that by \eqref{eq:first derivatives} 
\[
\Phi(t) = -\frac{2}{T} h_f'(t), \qquad \Psi(t) = \frac{2}{T} h_b'(t)
\]
with $h_f$, $h_b$ defined as in \eqref{eq:forward backward}, so that by Lemma \ref{lem:fg derivatives}
\[
\Phi'(t) \leq -\kappa T\Phi(t), \qquad \Psi'(t) \geq \kappa T\Psi(t).
\]
Hence, by Gr\"onwall's lemma
\begin{subequations}
\begin{align}
\label{eq:lower control 1}
& \Phi(t) \geq \exp(\kappa T(1/2-t))\Phi(1/2), \qquad \forall t \in [0,1/2], \\
\label{eq:lower control 2}
& \Psi(t) \geq \exp(\kappa T(t-1/2))\Psi(1/2), \qquad \forall t \in [1/2,1].
\end{align}
\end{subequations}
On the one hand, by \eqref{eq:cost representations} and \eqref{eq:lower control 2}
\[
\begin{split}
\cost(\mu_0^{T,\delta},\mu_1^{T,\delta}) - \cH(\mu_0^{T,\delta}\,|\,\mm) & \geq \frac{T}{2}\int_{1/2}^1 \Psi(t)\,\De t \geq \frac{T}{2}\Psi(1/2)\int_{1/2}^1 \exp(\kappa T(t-1/2))\,\De t \\
& = \frac{1}{2\kappa}\Psi(1/2)(\exp(\kappa T/2)-1),
\end{split}
\]
while on the other hand, by \eqref{eq:lower control 1} and a completely analogous argument
\[
\cost(\mu_0^{T,\delta},\mu_1^{T,\delta}) - \cH(\mu_1^{T,\delta}\,|\,\mm) \geq \frac{1}{2\kappa}\Phi(1/2)(\exp(\kappa T/2)-1).
\]
Combining these inequalities with \eqref{eq:phi psi}, we obtain
\[
|\cons(\mu_0^{T,\delta},\mu_1^{T,\delta})| \leq \frac{\kappa}{\exp(\kappa T/2)-1} \sqrt{(\cost(\mu_0^{T,\delta},\mu_1^{T,\delta})-\cH(\mu_0^{T,\delta}\,|\,\mm))(\cost(\mu_0^{T,\delta},\mu_1^{T,\delta})-\cH(\mu_1^{T,\delta}\,|\,\mm))}
\]
and it is now sufficient to pass to the limit as $\delta \downarrow 0$ to get the conclusion. The fact that $\cH(\mu_0^{T,\delta}\,|\,\mm) \to \cH(\mu\,|\,\mm)$ and $\cH(\mu_1^{T,\delta}\,|\,\mm) \to \cH(\nu\,|\,\mm)$ is motivated by dominated convergence; the same is true for $\cons(\mu_0^{T,\delta},\mu_1^{T,\delta}) \to \cons(\mu,\nu)$ taking \eqref{eq:energy delta} into account.
\end{proof}

\begin{remark}
{\rm 
Let us point out that the bound \eqref{eq:rate} is better than the one obtained from \eqref{eq:sharp} via a Taylor expansion around $\kappa = 0$, so that \eqref{eq:sharp} is not sharp and the reader might wonder whether \eqref{eq:sharp} could be improved by replacing $\exp(\kappa T/2)$ with $\exp(\kappa T)$. A posteriori, this is not possible because if \eqref{eq:sharp} held with $\exp(\alpha\kappa T)$ instead of $\exp(\kappa T/2)$ for some $\alpha > 1/2$, then this would contradict the sharpness of \eqref{long time conv bound} as stated in Theorem \ref{thm cost bound}.
}\fr
\end{remark}

We are now in the position to prove Theorem \ref{thm cost bound}.

\begin{proof}[Proof of Theorem \ref{thm cost bound}]
From Theorem \ref{thm:derivcost} and Theorem \ref{thm:longtime} we have that
\begin{equation}
\label{eq:vedo la luce}
|\cost(\mu,\nu) - \cH(\mu\,|\,\mm) - \cH(\nu\,|\,\mm)| = \Big|\int_T^\infty \mathscr{E}_S(\mu,\nu)\,\De S\Big| \leq \int_T^\infty |\mathscr{E}_S(\mu,\nu)|\,\De S
\end{equation}
and by Lemma \ref{lem:cons cost} the right-hand side above is controlled as follows
\begin{equation}
\label{eq:jeremy}
\int_T^\infty |\mathscr{E}_S(\mu,\nu)|\,\De S \leq \kappa \int_T^\infty \frac{\sqrt{\mathscr{C}_S(\mu,\nu)-\cH(\mu\,|\,\mm)}\sqrt{\mathscr{C}_S(\mu,\nu)-\cH(\nu\,|\,\mm)}}{\exp(\kappa S/2)-1}\, \De S.  
\end{equation}
By the entropic Talagrand inequality \eqref{eq:ent-talagrand} for $t=1/2$ and by \eqref{hyp pot} with $\kappa>0$, which implies $\cH(\cdot\,|\,\mm) \geq 0$, it holds
\[
\begin{split}
\mathscr{C}_S(\mu,\nu) - \cH(\mu\,|\,\mm) & \leq \exp(\kappa S/2)\frac{\exp(-\kappa S/2)\cH(\mu\,|\,\mm) + \cH(\nu\,|\,\mm)}{\exp(\kappa S/2)-1} \\
& \leq \exp(\kappa S/2)\frac{\exp(-\kappa T/2)\cH(\mu\,|\,\mm) + \cH(\nu\,|\,\mm)}{\exp(\kappa S/2)-1}
\end{split}
\]
for all $S \geq T$ and an analogous inequality holds with $\mu$ and $\nu$ swapped. Therefore the right-hand side in \eqref{eq:jeremy} is bounded from above by
\[
\kappa \sqrt{\exp(-\kappa T/2) \cH(\mu\,|\,\mm) + \cH(\nu\,|\,\mm)} \sqrt{\exp(-\kappa T/2) \cH(\nu\,|\,\mm) + \cH(\mu\,|\,\mm)} \int_T^\infty\frac{\exp(\kappa S/2)}{(\exp(\kappa S/2)-1)^2} \De S.
\]
The bound \eqref{eq:stronger bound 1} follows by calculating explicitly the integral term and this, in turn, trivially implies \eqref{long time conv bound}. 

For the sharpness of \eqref{long time conv bound}, we shall prove that there exists a triplet $(M',\sfd_g',\mm')$ satisfying \eqref{hyp pot} with $\kappa>0$ and $\mu,\nu \in \cP(M')$ satisfying \eqref{hyp marginals weak} such that
\begin{equation}
\label{eq:contradiction}
\lim_{T \to \infty}\frac{1}{T}\log|\cost(\mu,\nu) - \cH(\mu\,|\,\mm) - \cH(\nu\,|\,\mm)| \geq -\frac{\kappa}{2}.
\end{equation}
To this aim, consider $M' = \mathbb{R}$, $\sfd_g'$ the Euclidean distance and choose $U(x) := \kappa x^2/4$, so that $\mathrm{Hess}(2U) = \kappa$, $\mm'$ is the Gaussian measure
\begin{equation}
\label{eq:gaussian}
\mm'(\De x) = \sqrt{\frac{\kappa}{2\pi}}\exp\Big(-\frac{\kappa}{2}x^2\Big)\De x
\end{equation}
and the stochastic process associated to the SDE \eqref{SDE} is the stationary Ornstein-Uhlenbeck process, whence by well-known results (see e.g.\ \cite{BakryGentilLedoux14}) its transition probabilities w.r.t.\ $\mm'$ admit the following explicit representation
\begin{equation}
\label{eq:kernel}
\hp_t(x,y) = \frac{1}{\sqrt{1-\exp(-\kappa t)}}\exp\Big(-\frac{\kappa(x^2 - 2\exp(\kappa t/2)xy + y^2)}{2(\exp(\kappa t)-1)}\Big).
\end{equation}
With this said, choose $\mu,\nu \in \cP(\mathbb{R})$ with Lipschitz densities and compact supports with the further condition that
\[
\int x\,\mu(\De x) = -\int y\,\nu(\De y) = 1 \qquad \textrm{and} \qquad \int x^2\,\mu(\De x) = \int y^2\,\nu(\De y) = \sigma^2
\]
for some $\sigma > 0$. From the very definition of the entropic cost \eqref{cost def} and the identity $R_{0T}(\De x\De y) = \hp_T(x,y)\mm(\De x)\mm(\De y)$ we have
\[
\cost(\mu,\nu) = \inf_{\gamma \in \Pi(\mu,\nu)}\bigg\{ \cH(\gamma\,|\,\mm'\otimes\mm') - \int \log\hp_T\,\De\gamma \bigg\}
\]
and since $\int x^2\,\mu(\De x) = \int y^2\,\nu(\De y) = \sigma^2$ by construction, if we take also \eqref{eq:kernel} into account, then the above expression can be rewritten as
\[
\begin{split}
\cost(\mu,\nu) & = \inf_{\gamma \in \Pi(\mu,\nu)}\bigg\{ \cH(\gamma\,|\,\mm'\otimes\mm') - \frac{\kappa\exp(-\kappa T/2)}{2(1-\exp(-\kappa T))}\int xy\,\gamma(\De x\De y) \bigg\} \\
& \qquad + \underbrace{\frac{\kappa\sigma^2}{\exp(\kappa T)-1} + \frac{1}{2}\log(1-\exp(-\kappa T))}_{:= \beta(T)}.
\end{split}
\]
Now observe that there exists a unique optimal $\gamma^T \in \Pi(\mu,\nu)$ in the minimization problem above, as this is nothing but a rephrasing of \eqref{cost def}, and $\cH(\gamma^T\,|\,\mm' \otimes \mm') \geq \cH(\mu\,|\,\mm') + \cH(\nu\,|\,\mm')$, because $\mu \otimes \nu$ is the unique minimizer in \eqref{decoupledSP}. Moreover, by Lemma \ref{lem:gamma convergence} and Theorem \ref{thm:longtime} we know that $\gamma^T \rightharpoonup \mu \otimes \nu$ as $T \to \infty$ and, together with the fact that $\mu,\nu$ have compact support (whence a fortiori $\gamma^T$ in $\mathbb{R} \times \mathbb{R}$ too), this implies that
\[
\lim_{T \to \infty}\int xy\,\gamma^T(\De x\De y) = \int x\,\mu(\De x) \int y\,\nu(\De y),
\]
so that, thanks to the fact that $\int x\,\mu(\De x) = 1$ and $\int y\,\nu(\De y) = -1$, it holds $\int xy\,\gamma^T(\De x\De y) \leq -1/2$ for $T$ large enough. Finally, note that there exists $C>0$ such that $|\beta(T)| \leq C\exp(-\kappa T)$ for all $T$ sufficiently large. Combining all these ingredients, we deduce that
\[
\cost(\mu,\nu) - \cH(\mu\,|\,\mm') - \cH(\nu\,|\,\mm') \geq \frac{\kappa\exp(-\kappa T/2)}{4(1-\exp(-\kappa T))} - C\exp(-\kappa T)
\]
and it is now sufficient to remark that the right-hand side is asymptotically positive, so that for $T$ large enough we are allowed to apply the logarithm to both sides of the inequality above, whence \eqref{eq:contradiction}.

As regards \eqref{eq:long time energy bound}, it is sufficient to estimate the right-hand side in \eqref{eq:sharp} by following the same reasoning as above. To prove that also \eqref{eq:long time energy bound} is sharp, it is easy to see that if there exists $\alpha > 1/2$ such that
\[
\lim_{T \to \infty}\frac{1}{T}\log|\cons(\mu,\nu)| \leq -\alpha\kappa
\]
holds for all $\mu,\nu$ satisfying \eqref{hyp marginals weak}, then for any such $\mu,\nu$ there exist $C,T_0>0$ sufficiently large such that
\[
|\cons(\mu,\nu)| \leq C\exp(-\alpha\kappa T), \qquad \forall T \geq T_0
\]
and if we plug this inequality into \eqref{eq:vedo la luce}, instead of Lemma \ref{lem:cons cost}, then we get $|\cost(\mu,\nu) - \cH(\mu\,|\,\mm) - \cH(\nu\,|\,\mm)| \leq C\exp(-\alpha\kappa T)$, whence
\[
\lim_{T \to \infty}\frac{1}{T}\log|\cost(\mu,\nu) - \cH(\mu\,|\,\mm) - \cH(\nu\,|\,\mm)| \leq -\alpha\kappa
\]
and this clearly contradicts the sharpness of \eqref{long time conv bound}.
\end{proof}

\subsection{Proof of Theorem \ref{thm meanfieldlongtime}}

The proof relies on the results of \cite{backhoff2019mean} and we shall borrow some notation from the above mentioned work. In order to keep the size of the present work under control, we prefer not to reintroduce all the notation, but rather to give precise references to the place where this is done in \cite{backhoff2019mean}. In particular, Theorem 1.3 therein ensures that to any optimizer $\rmP$ for \eqref{nonlinearSP} there is an associated stochastic process $(\Psi_t(X_t))_{t\in[0,T]}$ such that 
\be
\label{eq:mart correct}
\Psi_t(X_t) - \int_{0}^t \tilde{\bbE}_{\tilde{\rmP}} \left[ \nabla^2 W(X_s-\tilde{X}_s)\cdot(\Psi_s(X_s) - \Psi_s(\tilde{X}_s) )\right] \, \De s
\ee
is a continuous martingale under $\rmP$ on $[0,T[$, where $(\tilde{X}_t)_{t\in[0,T]}$ is an independent copy of $(X_t)_{t\in[0,T]}$ and $\tilde{\bbE}_{\tilde{\rmP}}$ the corresponding expectation. In \cite[Lemma 4.5]{backhoff2019mean} it is proven that if $\rmP$ is an optimizer for \eqref{nonlinearSP}, then its \emph{time reversal} $\hat{\rmP}$, i.e. the law of $(X_{T-t})_{t\in[0,T]}$, is optimal for
\be
\label{time rev MFSP}
\inf \Big\{\cH(\rmQ\,|\,\Gamma(\rmQ) ) \,:\,  \rmQ \in \cP_{1}(\Omega),\,  \rmQ_0=\nu,\, \rmQ_T=\mu \Big\}.
\ee
The optimality of $\hat{\rmP}$ implies the existence of an associated process $\hat\Psi$ as in \eqref{eq:mart correct}, explicitly described by the relation
\begin{equation}
\label{eq:time reversal}
\hat\Psi_t(X_t) = -\Psi_{T-t}(X_t) + \nabla\log\rmP_{T-t}(X_t) + 2\nabla W\ast \rmP_{T-t}(X_t), \quad \hat{\rmP}\textrm{-a.s.}
\end{equation}
for all $t \in [0,T]$, see \cite[Eq. 92]{backhoff2019mean} and also \cite{follmer1986time}.

After this premise, the first ingredient needed for the proof of Theorem \ref{thm meanfieldlongtime} is the following

\begin{lemma}
Under the same assumptions as in Theorem \ref{thm meanfieldlongtime}, let $\rmP$ be optimal for \eqref{nonlinearSP}. Then
\bes
\mfcost(\mu,\nu) = \cF(\mu) +\cF(\nu) + \frac{1}{2}\bbE_{\rmP}\bigg[\int_0^{T/2}|\Psi_t(X_t)|^2\,\De t\bigg] + \frac{1}{2} \bbE_{\hat{\rmP}} \bigg[\int_0^{T/2}|\hat{\Psi}_t(X_t)|^2\,\De t\bigg] - \cF(\rmP_{T/2})
\ees
\end{lemma}

\begin{proof}
Using the time-reversal relation \eqref{eq:time reversal} we deduce that 
\[
\begin{split}
\frac{1}{2}\bbE_{\hat{\rmP}} \left[\int_{0}^{T/2} |\hat{\Psi}_t(X_t)|^2\,\De t\right] & = \frac{1}{2}\bbE_{\hat{\rmP}} \left[\int_{0}^{T/2} |\Psi_{T-t}(X_t)|^2\,\De t\right] \\
& \quad -\bbE_{\hat{\rmP}} \left[ \int_{0}^{T/2} (\Psi_{T-t}-\frac{1}{2}\chi_{T-t})\cdot \chi_{T-t}(X_t)\,\De t\right]
\end{split}
\]
where 
\bes
\chi_t := \nabla \log \rmP_t + 2\nabla W\ast \rmP_t .
\ees
Using the change of variable $s=T-t$ and the very definition of time reversal, we rewrite the right-hand side as
\bes
\frac{1}{2}\bbE_{\rmP} \left[\int_{T/2}^{T} |\Psi_{s}(X_s)|^2\,\De s\right]
-\bbE_{\rmP} \left[ \int_{T/2}^{T} (\Psi_{s}-\frac{1}{2}\chi_{s})\cdot \chi_{s})(X_s)\,\De s\right]
\ees
and using \cite[Lemma 4.4 (ii)]{backhoff2019mean} we find that the second member in the above expression is worth $-\cF(\nu)+\cF(\rmP_{T/2})$. Because of the mean field control formulation \cite[Lemma 1.1 and Theorem 1.3]{backhoff2019mean} of \ref{nonlinearSP} we have 
\bes
\mfcost(\mu,\nu)=\cF(\mu)+ \frac{1}{2}\bbE_{\rmP}\bigg[\int_0^{T}|\Psi_t(X_t)|^2 \De t\bigg],
\ees
so that the conclusion is easily obtained.
\end{proof}

As for the proof of Theorem \ref{thm cost bound}, also in this case we rely on an entropic Talagrand inequality. In the mean field setting (and taking into account Remark \ref{rem:factor2}) this reads as follows
\begin{equation}
\label{eq:mf-talagrand}
\mfcost(\mu,\nu) \leq \frac{1}{1-\exp(-\kappa Tt)} \cF(\mu) + \frac{1}{1-\exp(-\kappa T(1-t)) }\cF(\nu)
\end{equation}
for all $t \in (0,1)$ and $\mu,\nu$ satisfying \eqref{same mean ass}, see \cite[Corollary 1.3]{backhoff2019mean}.

We are now in position to prove Theorem \ref{thm meanfieldlongtime}.

\begin{proof}
From the previous lemma we get that 
\be
\label{eq:upperlowermfbound}
- \cF(\rmP_{T/2}) \leq \mfcost(\mu,\nu) - \cF(\mu) - \cF(\nu) \leq \frac{1}{2}\bbE_{\rmP} \bigg[\int_0^{T/2}|\Psi_t(X_t)|^2\,\De t\bigg] + \frac{1}{2}\bbE_{\hat{\rmP}}\bigg[\int_0^{T/2}|\hat{\Psi}_t(X_t)|^2\,\De t\bigg]. 
\ee
Now observe that, thanks to the fact that $\hat{\rmP}$ is optimal for \eqref{time rev MFSP}, both terms on the right-hand side in \eqref{eq:upperlowermfbound} can be bounded using first \cite[Lemma 4.3 Eq. 73]{backhoff2019mean} and then the entropic Talagrand inequality \eqref{eq:mf-talagrand} for $t=1/2$, keeping in mind Remark \ref{rem:factor2} and that the mean field entropic cost defined in this paper is related to the one defined in \cite{backhoff2019mean} via the identity \eqref{eq:mfcost}. We thus obtain, after some simple calculations
\bes
\frac{1}{2}\bbE_{\rmP}\bigg[\int_0^{T/2}|\Psi_t(X_t)|^2\,\De t\bigg] + \frac{1}{2} \bbE_{\hat{\rmP}} \bigg[\int_0^{T/2}|\hat{\Psi}_t(X_t)|^2\,\De t\bigg] \leq \frac{1}{\exp(\kappa T/2)-1 }(\cF(\mu)+\cF(\nu)),
\ees
which gives the desired upper bound for $\mfcost(\mu,\nu) - \cF(\mu) - \cF(\nu)$. For the lower bound, the conclusion directly follows from \cite[Theorem 1.4]{backhoff2019mean} and Remark \ref{rem:factor2}, which allow to estimate $\cF(\rmP_{T/2})$. Putting together the results obtained for the upper and lower bound we conclude.
\end{proof}

\appendix

\section{On the sharpness of the entropic Talagrand inequality}\label{appendix}

Under \eqref{hyp pot} with $\kappa>0$ it holds $\mm \in \cP_2(M)$, so that $\mm$ satisfies \eqref{hyp marginals weakest} and it is then licit to choose $\nu = \mm$ in \eqref{eq:ent-talagrand}, which gives rise to the following version of the entropic Talagrand inequality, closer to the classical Talagrand inequality known in optimal transport:
\begin{equation}
\label{eq:ent-talagrand2}
\cost(\mu,\mm) \leq \frac{1}{1-\exp(-\kappa T)}\cH(\mu\,|\,\mm)
\end{equation}
for all $\mu$ as in \eqref{hyp marginals weakest}. In \cite{Conforti17} neither the sharpness of \eqref{eq:ent-talagrand2} nor the one of \eqref{eq:ent-talagrand} was investigated. Aim of this appendix is to remedy this lack.

\begin{theorem}\label{lem:sharpness}
Assume that \eqref{hyp pot} holds with $\kappa>0$. Then the entropic Talagrand inequality \eqref{eq:ent-talagrand2} is sharp, since there exists a triplet $(M',\sfd_g',\mm')$ satisfying \eqref{hyp pot} with $\kappa>0$ such that
\begin{equation}
\label{eq:cost ent C}
\cost(\mu,\mm) \leq C\cH(\mu\,|\,\mm), \qquad \forall \mu \in \cP(M) \textrm{ satisfying \eqref{hyp marginals weakest}}
\end{equation}
does not hold for any
\[
C < \frac{1}{1-\exp(-\kappa T)}.
\]
As a byproduct, also the entropic Talagrand inequality \eqref{eq:ent-talagrand} is sharp in the following sense: there exists a triplet $(M',\sfd_g',\mm')$ satisfying \eqref{hyp pot} with $\kappa>0$ such that
\[
\begin{split}
\cost(\mu,\nu) \leq \frac{1}{1-\exp(-\alpha Tt)}\cH(\mu\,|\,\mm) & + \frac{1}{1-\exp(-\alpha T(1-t))}\cH(\nu\,|\,\mm), \\
& \qquad \forall t \in (0,1),\forall \mu,\nu \in \cP(M) \textrm{ satisfying \eqref{hyp marginals weakest}}
\end{split}
\]
does not hold for any $\alpha > \kappa$.
\end{theorem}

\begin{proof}
As a first step, we shall prove that \eqref{eq:cost ent C} is equivalent to
\begin{equation}
\label{eq:exp exp}
\int\exp\Big(-\frac{1}{TC}\phi\Big)\,\De\mm \leq \exp\Big(-\frac{1}{TC}\int \sfQ^T_1\phi\,\De\mm\Big), \qquad \forall \phi \in C_b(M),
\end{equation}
where $\sfQ_t^T\phi := -T\log\sfP_{Tt}(\exp(-\phi/T))$. To this aim, notice that by \eqref{eq:dual representation} and the symmetry of the entropic cost, i.e.\ $\cost(\mu,\mm) = \cost(\mm,\mu)$, \eqref{eq:cost ent C} is equivalent to
\begin{equation}
\label{eq:kill bill}
\int\sfQ_1^T\phi\,\De\mm - \int\phi\,\De\mu - TC\cH(\mu\,|\,\mm) \leq 0
\end{equation}
for all $\mu \in \cP(M)$ satisfying \eqref{hyp marginals weakest} and for all $\phi \in C_b(M)$. Then observe that
\[
\log\int\exp(\psi)\,\De\mm = \sup\bigg\{\int\psi\,\De\mu - \cH(\mu\,|\,\mm) \,:\, \mu \in \cP(M),\, \int\psi^+\,\De\mu < \infty \bigg\},
\]
where $\psi^+ := \max\{\psi,0\}$, the identity being a byproduct of the well-known variational representation of the entropy (see \cite{GozlanLeonard10,Leonard14b}). Hence by taking the supremum over $\mu$ in \eqref{eq:kill bill}, we get the equivalent formulation
\[
\int\sfQ_1^T\phi\,\De\mm + TC\log\int\exp\Big(-\frac{1}{TC}\phi\Big) \,\De\mm \leq 0, \qquad \forall \phi \in C_b(M),
\]
which is clearly equivalent to \eqref{eq:exp exp} by algebraic manipulations.

Relying on this characterization, we are now in the position to prove the sharpness of \eqref{eq:ent-talagrand2}. As in the proof of Theorem \ref{thm cost bound}, consider the real line endowed with the Euclidean distance and choose $U(x) := \kappa x^2/4$, so that $\mathrm{Hess}(2U) = \kappa$, $\mm$ is the Gaussian measure defined in \eqref{eq:gaussian} and the stochastic process associated to the SDE \eqref{SDE} is the stationary Ornstein-Uhlenbeck process, whence the explicit representation for its transition probabilities $\hp_t(x,y)$ described in \eqref{eq:kernel}. With this said, assume that \eqref{eq:cost ent C} holds for some constant $C>0$, hence \eqref{eq:exp exp} as well. By standard approximation arguments, in \eqref{eq:exp exp} $C_b(\mathbb{R})$ can be replaced by continuous functions with at most linear growth at infinity, so that
\begin{equation}
\label{eq:sojuz}
\int\exp\Big(-\frac{1}{TC}\phi\Big)\,\De\mm \leq \exp\Big(-\frac{1}{TC}\int \sfQ^T_1\phi\,\De\mm\Big)
\end{equation}
holds for $\phi(x) := \alpha x$. On the one hand
\[
\int\exp\Big(-\frac{1}{TC}\phi\Big)\,\De\mm = \exp\Big(\frac{\alpha^2}{2\kappa T^2 C^2}\Big),
\]
while on the other hand
\[
\sfQ_1^T\phi(x) = -T\log\int\exp\Big(-\frac{\alpha}{T}y\Big) \hp_T(x,y)\,\mm(\De y) = x\alpha\exp\Big(-\frac{\kappa T}{2}\Big) - \alpha^2\frac{1-\exp(-\kappa T)}{2\kappa T},
\]
so that
\[
\exp\Big(-\frac{1}{TC}\int \sfQ^T_1\phi\,\De\mm\Big) = \exp\Big(\alpha^2\frac{1-\exp(-\kappa T)}{2\kappa T^2 C}\Big).
\]
Combining these identities with \eqref{eq:sojuz} we obtain
\[
\frac{\alpha^2}{2\kappa T^2 C^2} \leq \alpha^2\frac{1-\exp(-\kappa T)}{2\kappa T^2 C}
\]
and this inequality is satisfied if and only if $C \geq (1-\exp(-\kappa T))^{-1}$, as claimed.

The sharpness of \eqref{eq:ent-talagrand} immediately follows from the one of \eqref{eq:ent-talagrand2}. Indeed, if there exists $\alpha > \kappa$ such that
\[
\cost(\mu,\nu) \leq \frac{1}{1-\exp(-\alpha Tt)}\cH(\mu\,|\,\mm) + \frac{1}{1-\exp(-\alpha T(1-t))}\cH(\nu\,|\,\mm)
\]
holds for all $\mu,\nu \in \cP(M)$ satisfying \eqref{hyp marginals weakest}, then by choosing $\nu = \mm$ and letting $t \to 1$ we get a contradiction.
\end{proof}

\paragraph{Acknowledgements.} The second author gratefully acknowledges support by the European Union through the ERC-AdG ``RicciBounds'' for Prof.\ K.\ T.\ Sturm.

\bibliographystyle{siam}
{\small
\bibliography{biblio}}

\end{document}